\documentclass[ams,english]{amsart}
\usepackage[latin9]{inputenc}
\usepackage{amsmath}

\usepackage{bbm}
\usepackage{babel}
\usepackage{verbatim}
\usepackage{mathtools}
\usepackage{amsthm}
\usepackage{amstext}
\usepackage{amssymb}
\usepackage{amsfonts}
\usepackage{esint}
\usepackage[unicode=true,
bookmarks=false,
breaklinks=false,pdfborder={0 0 1},backref=false,colorlinks=false]
{hyperref}
\hypersetup{
	colorlinks,linkcolor=blue,anchorcolor=blue,citecolor=blue}
\usepackage{breakurl}
\usepackage{mathrsfs}

\usepackage{amsmath}

\makeatletter

\numberwithin{equation}{section}
\numberwithin{figure}{section}
\usepackage{enumitem}		

\theoremstyle{plain}
\newtheorem{thm}{Theorem}[section]
\theoremstyle{plain}
\newtheorem{prop}[thm]{Proposition}
\theoremstyle{remark}
\newtheorem{rem}[thm]{Remark}

\theoremstyle{plain}
\newtheorem{cor}[thm]{Corollary}
\theoremstyle{plain}

\theoremstyle{plain}
\newtheorem{lem}[thm]{Lemma}
\theoremstyle{definition}

\theoremstyle{definition}
\newtheorem{defn}[thm]{Definition}

\newcommand{\Tr}{\mathrm{Tr}}
\newcommand{\Pol}{\mathrm{Pol}}
\newcommand{\Irr}{\mathrm{Irr}}

\usepackage{accents}
\newlength{\dhatheight}
\newcommand{\doublehat}[1]{%
	\settoheight{\dhatheight}{\ensuremath{\hat{#1}}}%
	\addtolength{\dhatheight}{-0.35ex}%
	\hat{\vphantom{\rule{1pt}{\dhatheight}}%
		\smash{\hat{#1}}}}

\begin{document}
	
	\title{Infinitely divisible states on finite quantum groups}
	
	\author{Haonan Zhang}
	
	\address{Laboratoire de Math\'ematiques, Universit\'e de Franche-Comt\'e, 25030 Besan\c con, France and Institute of Mathematics, Polish Academy of Sciences, ul. \'Sniadeckich 8, 00-956 Warszawa, Poland}
	
	\email{haonan.zhang@edu.univ-fcomte.fr}
	
	
	
	\begin{abstract}
		In this paper we study the states of Poisson type and infinitely divisible states on compact quantum groups. Each state of Poisson type is infinitely divisible, i.e., it admits $n$-th root for all $n\geq1$. The main result is that on finite quantum groups infinitely divisible states must be of Poisson type. This generalizes B\"oge's theorem concerning infinitely divisible measures (commutative case) and Parthasarathy's result on infinitely divisible positive definite functions (cocommutative case). Two proofs are given.
	\end{abstract}
	
	\maketitle
	
	\section*{0. Introduction}
	The space of bounded measures on a compact (semi)group is equipped with a natural convolution operation. The convolution of two probability measures is still a probability measure. Infinitely divisible probability measures are probability measures that admit $n$-th root for all $n\geq1$, where the root is also a probability measure. On finite groups such probability measures have been shown to be of Poisson type, see \cite{Boge} and \cite{poissonlaws}.
	
	A positive definite function on a compact group $G$ is a continuous function $\phi:G \to \mathbb{C}$ such that $[\phi(g^{-1}_ig_j)]_{i,j=1}^{n}$ is a positive semi-definite matrix for all $g_1,\dots,g_n\in G$ and for all $n\geq 1$. It is normalized if $\phi(e)=1$, where $e$ is the unit of $G$. The pointwise product of two normalized positive definite functions on $G$ is again a normalized positive definite function. From this we can define infinitely divisible normalized positive definite functions on a compact group in a natural way. This is thoroughly studied by Parthasarathy \cite{KRP}. As a special case, he proved that every infinitely divisible normalized positive definite function on a finite group is of Poisson type, although the notion ``Poisson type" was not explicitly defined in his paper. 
	
	We shall consider the infinite divisibility of states on quantum groups, which provide a more general framework. Our main result is that any infinitely divisible state on a finite quantum group is \emph{of Poisson type} (in the following also called simply a \emph{Poisson state}). By taking the finite quantum group to be commutative and cocommutative, we recover the B\"oge's result \cite{Boge} of infinitely divisible probability measures and Parthasarathy's result \cite{KRP} on infinitely divisible normalized positive definite functions for finite groups, respectively. We will give two proofs of the main theorem. The first one is based on the ideas of \cite{poissonlaws} and the second one goes back to \cite{KRP}. 
	
	The main difficulty in the study of infinite divisibility of states on quantum groups is to capture the ``quantum subgroup" on which the states are ``supported". It is known that the notion ``quantum subgroup" here should be replaced by ``quantum hypersubgroup" \cite{Idempotentstatesoncqg}. Indeed, this is closely related to idempotent states. On a classical compact group, idempotent probability measures are Haar measures on compact subgroups, due to Kawada and It\^o \cite{KawadaIto}. In quantum group case, idempotent states not necessarily correspond to quantum subgroups. See \cite{Pal} for the first counter example on 8-dimensional Kac-Paljutkin quantum group; more discussions can be found in \cite{Idempotentstatesoncqg}. The right concept one should consider here is the quantum hypergroup \cite{Idempotentstatesoncqg}. That is first reason why a whole section (Section 2) is devoted to the study of compact quantum hypergroups. The second reason for devoting much effort to compact quantum hypergroups here is the fact that, compared with the theory of compact quantum groups, very little is known for compact quantum hypergroups. The original definition \cite{Compactquantumhypergroup} is rather technical, which makes it difficult to construct examples. Thus relatively few concrete compact quantum hypergroups are known so far \cite{Compactquantumhypergroup,Newexamples}. This motivates us to present two approaches to constructions of compact quantum hypergroups in Section 2, to enlarge the class of relevant examples. We mention here that although the theory of algebraic quantum groups developed by Delvaux and Van Daele \cite{Algebraicquantumhypergroups} is very nice, it can not serve as a substitute of compact quantum hypergroups here; see Section 1 and Section 2 for more discussions.
	
	The plan of this paper is as follows. In Section 1 we recall the preliminaries on compact quantum groups, compact quantum hypergroups and algebraic quantum hypergroups. In Section 2 we give two approaches to construct compact quantum hypergroups from compact quantum groups, one induced by an idempotent state and one from a group-like projection. We also give a duality theorem for finite quantum hypergroups. Part of results in this section are new. Section 3 is devoted to the study of Poisson states on compact quantum groups. Finally in Section 4 we prove the main result of this paper, namely  that any infinitely divisible state on a finite quantum group is a Poisson state, in two different ways.
	
	\section{Preliminaries}
	\subsection{Compact quantum group and its dual}
	Let us recall some definitions and properties of compact quantum groups. We refer to \cite{CompactquantumgroupsWoronowicz} and \cite{CompactquantumgroupsVanDaele} for more details.
	
	\begin{defn}
		Let $A$ be a unital $C^*$-algebra. If there exists a unital *-homomorphism $\Delta:A\to A\otimes A$ such that 
		\begin{enumerate}
			\item $(\Delta\otimes\iota)\Delta=(\iota\otimes\Delta)\Delta$;
			\item $\{\Delta(a)(1\otimes b):a,b\in A\}$ and $\{\Delta(a)(b\otimes1):a,b\in A\}$ are linearly dense in $A\otimes A$,
		\end{enumerate}
		then $(A,\Delta)$ is called a \textit{compact quantum group} and $\Delta$ is called the
		comultiplication on $A$. Here and in the following, $\iota$ always denotes the identity map. We denote $\mathbb{G}=(A,\Delta)$ and $A=C(\mathbb{G})$. For simplicity, we write $\Delta^{(2)}=(\Delta\otimes\iota)\Delta$.
	\end{defn}
	
	Any compact quantum group $\mathbb{G}=(A,\Delta)$ admits a unique \textit{Haar state}, i.e.\ a state $h$ on $A$ such that
	\[
	(h\otimes\iota)\Delta(a)=h(a)1=(\iota\otimes h)\Delta(a),~~a\in A.
	\]
	Consider an element $u\in A\otimes B(H)$ with $\dim H=n$. By identifying $A\otimes B(H)$ with $\mathbb{M}_n(A)$ we can write $u=[u_{ij}]_{i,j=1}^{n}$. The matrix $u$ is called an \textit{n-dimensional representation} of $\mathbb{G}$ if we have
	\[
	\Delta(u_{ij})=\sum_{k=1}^{n}u_{ik}\otimes u_{kj},~~i,j=1,\dots,n.
	\]
	A representation $u$ is called \textit{unitary} if $u$ is unitary as an element in $\mathbb{M}_n(A)$, and \textit{irreducible} if the only matrices $T\in\mathbb{M}_n(\mathbb{C})$ such that $uT=Tu$ are multiples of identity matrix. Two representations $u,v\in\mathbb{M}_n(A)$ are said to be \textit{equivalent} if there exists an invertible matrix $T\in\mathbb{M}_n(\mathbb{C})$ such that $Tu=vT$. Denote by $\Irr(\mathbb{G})$ the set of equivalence classes of irreducible unitary representations of $\mathbb{G}$. For each $\alpha\in\Irr(\mathbb{G})$, denote by $u^\alpha\in A\otimes B(H_\alpha)$ a representative of the class $\alpha$, where $H_\alpha$ is the finite dimensional Hilbert space on which $u^\alpha$ acts. In the sequel we write $n_\alpha=\dim H_\alpha$.
	
	Denote $\Pol(\mathbb{G})=\text{span} \left\{u^\alpha_{ij}:1\leq i,j\leq n_\alpha,\alpha\in\Irr(\mathbb{G})\right\}$. This is a dense subalgebra of $A$. It is well-known that $(\Pol(\mathbb{G}),\Delta)$ is equipped with the Hopf*-algebra structure. That is, there exist a linear antihomormophism $S$ on $\Pol(\mathbb{G})$, called the \textit{antipode}, and a unital *-homomorphism $\epsilon:\Pol(\mathbb{G})\to\mathbb{C}$, called the \textit{counit}, such that
	\[
	(\epsilon\otimes\iota)\Delta(a)=a=(\iota\otimes\epsilon)\Delta(a),~~a\in\Pol(\mathbb{G}),
	\]
	\[
	m(S\otimes\iota)\Delta(a)=\epsilon(a)1=m(\iota\otimes S)\Delta(a),~~a\in\Pol(\mathbb{G}).
	\]
	Here $m$ denotes the multiplication map $m:\Pol(\mathbb{G})\otimes_{\text{alg}}\Pol(\mathbb{G})\to\Pol(\mathbb{G}),~~a\otimes b\mapsto ab$.
	Indeed, the antipode and the counit are uniquely determined by
	\[
	S(u^\alpha_{ij})=(u^{\alpha}_{ji})^*,~~1\leq i,j\leq n_\alpha,~~\alpha\in\Irr(\mathbb{G}),
	\]
	\[
	\epsilon(u^\alpha_{ij})=\delta_{ij},~~1\leq i,j\leq n_\alpha,~~\alpha\in\Irr(\mathbb{G}).
	\]
	Remark here that $*\circ S\circ*\circ S=\iota$.
	
	The Peter-Weyl theory for compact groups can be extended to the quantum case. In particular, it is known that for each $\alpha\in\Irr(\mathbb{G})$ there exists a positive invertible operator $Q_\alpha\in B(H_\alpha)$ such that $\text{Tr}(Q_\alpha)=\text{Tr}(Q^{-1}_\alpha):=d_\alpha$, which we call \textit{quantum dimension} of $\alpha$, and the orthogonal relations hold:
	\[
	h(u^\alpha_{ij}(u^\beta_{kl})^*)=\frac{\delta_{\alpha\beta}\delta_{ik}(Q_\alpha)_{lj}}{d_\alpha},~~ 
	h((u^\alpha_{ij})^*u^\beta_{kl})=\frac{\delta_{\alpha\beta}\delta_{jl}(Q^{-1}_\alpha)_{ki}}{d_\alpha},
	\]
	where $\alpha,\beta\in\Irr(\mathbb{G})$, $1\leq i,j\leq n_\alpha, 1\leq k,l\leq n_\beta.$
	
	The Pontryagin duality can also be extended to compact quantum groups. The dual quantum group $\hat{\mathbb{G}}$ of $\mathbb{G}=(A,\Delta)$ is defined via its ``algebra of functions", which is the $C^*$-algebra defined as the $c_0$-direct sum
	\[
	\hat{A}=\bigoplus_{\alpha\in\Irr(\mathbb{G})}B(H_\alpha).
	\] 
	Unless $\mathbb{G}$ is finite quantum group, $\hat{A}$ is not unital. We define $\hat{\mathcal{A}}$ as the *-algebra via the algebraic direct sum
	\[
	\hat{\mathcal{A}}=\bigoplus_{\alpha\in\Irr(\mathbb{G})}B(H_\alpha).
	\]
	That is to say, each element of $\hat{\mathcal{A}}$ has only finitely many non-zero components in the direct summands. Clearly, $\hat{\mathcal{A}}$ is dense in $\hat{A}$.
	
	We can equip $\hat{A}$ with a \textit{discrete quantum group} structure. See \cite{Discretequantumgroups} for more details. Recall only that the \textit{left Haar weight} $\hat{h}_{L}$ on $\hat{\mathbb{G}}$ is given by
	\[
	\hat{h}_L(x)=\sum_{\alpha\in\Irr(\mathbb{G})}\Tr(Q_\alpha)\Tr(Q_\alpha x_\alpha),~~x\in\hat{A},
	\]
	where $x_\alpha$ denotes the component of $x$ in the direct summand $B(H_\alpha)$. The \textit{right Haar weight} $\hat{h}_{R}$ shares a similar form. 
	
	On $\Pol(\mathbb{G})^*$, the set of bounded linear functionals on $\Pol(\mathbb{G})$ (where $\Pol(\mathbb{G})$ is viewed with the universal enveloping $C^*$-norm), there is a natural Banach *-algebra structure. When $\varphi_1,\varphi_2$ are two bounded linear functionals on $\Pol(\mathbb{G})$, then their \textit{convolution product} is defined as $\varphi_1\star\varphi_2:=(\varphi_1\otimes\varphi_2)\Delta$. Clearly we have $\Vert\varphi_1\star\varphi_2\Vert\leq\Vert\varphi_1\Vert\Vert\varphi_2\Vert$, where the norm of functionals on $\Pol(\mathbb{G})$ always come from $C(\mathbb{G})^*$. For any $\varphi\in\Pol(\mathbb{G})^*$, define the involution of $\varphi$ as $\varphi^*:=\overline{\varphi(S(\cdot)^*)}$. One can also construct the dual of $\mathbb{G}=(A,\Delta)$ via the functionals on $A$ with this *-algebra structure, see \cite{CompactquantumgroupsVanDaele}. Here we use the \textit{Fourier transform} to say a few words on this.
	
	For a linear functional $\varphi$ on $\Pol(\mathbb{G})$, define its Fourier transform $\hat{\varphi}=(\hat{\varphi}(\alpha))_{\alpha\in\Irr(\mathbb{G})}\in\bigoplus_{\alpha\in\Irr(\mathbb{G})}B(H_\alpha)$ by
	\[
	\hat{\varphi}(\alpha)=(\varphi\otimes\iota)(u^\alpha)\in B(H_\alpha),~~\alpha\in\Irr(\mathbb{G}).
	\]
	The Fourier transform $\mathcal{F}:\varphi\mapsto\hat{\varphi}$ sends the convolution to multiplication 
	\[
	(\varphi_1\star\varphi_2)\string^(\alpha)=\hat{\varphi_1}(\alpha)\hat{\varphi_2}(\alpha),~~\alpha\in\Irr(\mathbb{G}),\varphi_1,\varphi_2\in\Pol(\mathbb{G})^*,
	\] 
	and is *-preserving: $\widehat{\varphi^*}=\hat{\varphi}^*,\varphi\in\Pol(\mathbb{G})^*$. Moreover, $\Vert\hat{\varphi}\Vert=\Vert(\varphi\otimes\iota)(W)\Vert\leq1$, where $W=\oplus_{\alpha\in\Irr(\mathbb{G})}u^{\alpha}$ is unitary. So $\mathcal{F}$ is a contraction.
	
	We call $\mathbb{G}$ a \textit{finite quantum group} if the underlying $C^*$-algebra $C(\mathbb{G})$ is finite dimensional. In this case each $Q_\alpha$ is identity and $h$ is also a trace, i.e., $h(ab)=h(ba)$ for any $a,b\in C(\mathbb{G})$. Then the orthogonal relation becomes
	\begin{equation}\label{Orthogonal relation for finite quantum groups}
	h(u^\alpha_{ij}(u^\beta_{kl})^*)=h((u^\alpha_{ij})^*u^\beta_{kl})=\frac{\delta_{\alpha\beta}\delta_{ik}\delta_{jl}}{n_\alpha},
	\end{equation}
	where $\alpha,\beta\in\Irr(\mathbb{G})$, $1\leq i,j\leq n_\alpha, 1\leq k,l\leq n_\beta.$ If $\mathbb{G}$ is a finite quantum group, then so is its dual $\hat{\mathbb{G}}$, and the corresponding Haar weights $\hat{h}_{L}$ and $\hat{h}_{R}$ coincide, and are denoted by $\hat{h}$. After normalization the functional $\hat{h}$ becomes a tracial state of the form:
	\[
	\hat{h}(x)=\frac{1}{h(1)}\sum_{\alpha\in\Irr(\mathbb{G})}n_\alpha\Tr(x_\alpha),~~x\in\hat{A}.
	\]
	Moreover, the antipode $S$ satisfies $S^2=\iota$. Together with $*\circ S\circ *\circ S=\iota$, one obtains directly that $S$ is *-preserving. The Fourier transform $\mathcal{F}$ now is a *-isomorphism between the C*-algebras $(A,\cdot,*)$ and $(A^*,\star,*)$. The notation $\hat{h}$ has some conflict with the Fourier transform of $h$. It will not be difficult for readers to distinguish them by observing the elements they act on. 
	
	\subsection{Compact quantum hypergroups and *-algebraic quantum hypergroups}
	Compact quantum hypergroups were introduced by Chapovsky and Vainerman in \cite{Compactquantumhypergroup}. Their definition  is very technical, relying on the existence of a one-parameter group of automorphisms verifying certain relations. This brings a lot of trouble constructing non-trivial compact quantum hypergroups. Later on Kalyuzhnyi proposed \cite{Newexamples} a construction of compact quantum hypergroups using conditional expectations on compact quantum groups. The compact quantum hypergroups discussed in this paper mainly come from this construction. However, we will also give an improvement of Kalyuzhnyi's result in the sense that, a new class of conditional expectations that do not verify Kalyuzhnyi's conditions but still induce compact quantum hypergroups, is constructed. Indeed, such constructions have already been studied by Delvaux and Van Daele on the algebraic level and have also been widely used by others, see for example \cite{Idempotentstatesoncqg} and \cite{Newcharacterisationofidempotentstates}. 
	
	In \cite{Algebraicquantumhypergroups} Delvaux and Van Daele introduced the so-called \emph{*-algebraic quantum hypergroups}, which are essentially the algebraic counterparts of compact quantum hypergroups. In a separate note \cite{constructionsandexamplesofquantumhypergroups} they gave several constructions and examples of *-algebraic quantum hypergroups. Note that even if it is reasonable to expect that a *-algebraic quantum hypergroup of compact type with positive integrals will yield a compact quantum hypergroup in the sense of \cite{Compactquantumhypergroup} (see \cite{Algebraicquantumhypergroups}), it has not been shown so far. We hope that this could be established in the future.
	
	Delvaux and Van Daele's *-algebraic quantum hypergroups admit a very nice biduality theory \cite[Theorem 3.12]{Algebraicquantumhypergroups}, that is, for a *-algebraic quantum hypergroup $(A,\Delta)$ with its dual $(\hat{A},\hat{\Delta})$, there is a natural isomorphism between $(A,\Delta)$ and its bidual $(\doublehat{A},\doublehat{\Delta})$.  We will however use mainly compact quantum hypergrous in this paper. The reason for this is that, on compact quantum hypergroups a representation theory similar to that of compact quantum groups was developed \cite{Compactquantumhypergroup}. See Theorem \ref{Peter-Weyl quantum hyper group} below for a Peter-Weyl theory for compact quantum hypergroups.
	
	\subsubsection{Compact quantum hypergroups}
	Now we introduce the compact quantum hypergroups. Let $(A,\cdot,1,*)$ be a separable unital C*-algebra. We will call $(A,\delta,\epsilon,*)$ a \emph{hypergroup structure} on the C*-algebra $(A,\cdot,1,*)$ if $(A,\delta,\epsilon,\star)$ is a $\star$-coalgebra in the following sense:
	\begin{enumerate}
		\item $\delta:A\to A\otimes A$ is a linear positive map such that
		$$(\delta\otimes\iota)\delta=(\iota\otimes\delta)\delta,$$
		$$\delta(a^*)=\delta(a)^*,~~a\in A,$$
		$$\delta(1)=1\otimes 1;$$
		\item $\epsilon:A\to\mathbb{C}$ is a linear map such that
		$$\epsilon(ab)=\epsilon(a)\epsilon(b),~~a,b\in A,$$
		$$(\epsilon\otimes\iota)\delta=(\iota\otimes\epsilon)\delta=\iota;$$
		\item $\star:A\to A$ is an anti-linear map such that 
		$$\star\circ\star=\iota,$$
		$$(ab)^\star=a^\star b^\star,~~a,b\in A,$$
		$$\star\circ*=*\circ\star,$$ 
		$$\delta\circ\star=\Pi\circ(\star\otimes\star)\circ\delta,$$
		where $\Pi$ is the flip on $A\otimes A$, i.e., $\Pi(a\otimes b)=b\otimes a$.
	\end{enumerate}
	Under these assumptions one can deduce that \cite[Lemma 1.2]{Compactquantumhypergroup}
	$$1^\star=1,$$
	$$\epsilon(1)=1,$$
	$$\epsilon(a^\star)=\epsilon(a^*)=\overline{\epsilon(a)},~~a\in A.$$
	Denote by $A^*$ the set of all bounded linear continuous functionals on $A$, then for $\xi,\eta\in A^*$ we can define a product and an involution $^+$
	$$(\xi\cdot\eta)(a):=(\xi\otimes\eta)\delta(a),~~a\in A,$$
	$$\xi^+(a):=\overline{\xi(a^\star)},~~a\in A.$$
	Moreover, we can equip $A^*$ with the following norm
	$$\Vert\xi\Vert:=\sup\{|\xi(a)|:\Vert a\Vert\leq 1\}.$$
	Thus $(A^*,\cdot,+,\Vert\cdot\Vert)$ becomes a Banach $^+$-algebra \cite[Lemma 1.3]{Compactquantumhypergroup}.
	
	Let $A$ be a unital C*-algebra equipped with a hypergroup structure as above. An element $a\in A$ is called \emph{positive definite} if $\xi\cdot\xi^+(a)\geq0$ for all $\xi\in A^*$. It is known that \cite[Theorem 2.3]{Compactquantumhypergroup} if the linear span of all positive definite elements in $A$ is dense in $A$, then there exists a unique $h\in A^*$ such that 
	$$(h\otimes\iota)\delta(a)=(\iota\otimes h)\delta(a)=h(a)1,~~a\in A,$$
	called the \emph{Haar measure} on $A$. Moreover, $h=h^+.$
	
	Now we are ready to introduce the \emph{compact quantum hypergroup}.
	
	\begin{defn}\cite[Definition 4.1]{Compactquantumhypergroup}
		Suppose that $(A,\delta,\epsilon,*)$ is a hypergroup structure on the C*-algebra $(A,\cdot,1,*)$. We call $\mathcal{A}=(A,\cdot,1,*,\delta,\epsilon,\star,\sigma_t)$ a \emph{compact quantum hypergroup} if 
		\begin{enumerate}
			\item the linear span of positive definite elements of $A$ is dense in $A$ (thus there exists the unique Haar measure $h$);
			\item $\delta$ is completely positive;
			\item $(\sigma_t)_{t\in\mathbb{R}}$ is a continuous one-parameter group of automorphisms of $A$ such that there exists dense subalgebras $A_0\subset A$ and $\widetilde{A}_0\subset A\otimes A$ verifying 
			\begin{enumerate}
				\item The one-parameter groups $\sigma_t$, $\iota\otimes\sigma_t$ and $\sigma_t\otimes\iota$ can be extended to complex one-parameter groups $\sigma_z$, $\iota\otimes\sigma_z$ and $\sigma_z\otimes\iota$, $z\in\mathbb{C}$, of automorphisms of $A_0$ and $\widetilde{A}_0$, respectively;
				\item $A_0$ is invariant under $*$ and $\star$, and $\delta(A_0)\subset\widetilde{A}_0$;
				\item  for all $z\in\mathbb{C}$ and $a \in A_0$ there is
				$$\delta\sigma_z(a)=(\sigma_z\otimes\sigma_z)\delta(a),$$
				$$h(\sigma_z(a))=h(a);$$
				\item there exists $z_0\in\mathbb{C}$ such that $h$ satisfies the \emph{strong invariance condition}:
				\begin{equation}\label{strong invariance}
				(\iota\otimes h)[((*\circ\sigma_{z_0}\circ\star)\otimes\iota)\delta(a)(1\otimes b)]=(\iota\otimes h)[(1\otimes a)\delta(b)],
				\end{equation}
				for all $a,b\in A_0.$
				\item $h$ is faithful on $A_0$. 
			\end{enumerate} 
		\end{enumerate}
	\end{defn}
	
	For short we write $\mathbb{H}=(A,\delta)$ to denote the compact quantum hypergroup. We can also define the antipode as $\kappa:=*\circ\sigma_{z_0}\circ\star$. It is invertible with the inverse $\kappa^{-1}=\star\circ\sigma_{-z_0}\circ*$. See \cite[Lemma 4.4]{Compactquantumhypergroup} for more properties of $\kappa$.
	
	As one can see, the definition of compact quantum hypergroup is very complicated and technical. Certainly compact quantum groups and the classical compact hypergroups are compact quantum hypergroups \cite{Compactquantumhypergroup}. But usually it is very difficult to construct other examples. One way to do this uses a sufficiently nice conditional expectation on a compact quantum group, as the following theorem shows.
	
	\begin{thm}\cite{Newexamples}\label{New examples}
		Let $\mathbb{G}=(A,\Delta,\epsilon,S)$ be a compact quantum group. Let $h$ be its Haar measure and $P:A\to B$
		be an $h$-invariant conditional expectation that maps to a unital $C^*$-subalgebra
		$B$ of $A$. Let us define a new comultiplication $\widetilde{\Delta}$ on $B$ as $\widetilde{\Delta}=(P\otimes P)\Delta|_B$. Suppose that 
		\begin{enumerate}
			\item $(P\otimes P)\Delta(x)=(P\otimes P)\Delta(P(x)),x\in A$;
			\item $\Pol(\mathbb{G})$ is invariant under $P$ and $S P=P S$; 
			\item the restriction of $\epsilon$ to $B$ is a counit, i.e.,
			\[
			(\epsilon\otimes\iota)\widetilde{\Delta}=\iota=(\iota\otimes\epsilon)\widetilde{\Delta}.
			\]
		\end{enumerate}
		Then $(B,\widetilde{\Delta})$ forms a compact quantum hypergroup.
	\end{thm}
	
	Like for compact quantum groups, there is a representation theory for compact quantum hypergroups. A matrix $u=[u_{ij}]_{i,j=1}^{n}\in M_n(A)$ is called an \textit{n-dimensional representation} of $A$ if we have
	$$\delta(u_{ij})=\sum_{k=1}^{n}u_{ik}\otimes u_{kj},$$
	$$\epsilon(u_{ij})=\delta_{ij},$$
	for all $i,j=1,\dots,n$. Here $\delta_{ij}$ denotes the Kronecker symbol. It is called a $^\dagger$-representation if $u=[u_{ij}]_{i,j=1}^{n}\in M_n(A_0)$ and $u^\dagger_{ij}=u_{ji},1\leq i,j\leq n$, where $a^\dagger:=\kappa(a)^*$ for $a\in A_0$.
	
	Let $\Irr(\mathbb{H})$ denote a maximal set of finite dimensional irreducible non-equivalent $^\dagger$-representations $u^\alpha=[u^\alpha_{ij}]_{i,j=1}^{n_\alpha}$. We write $\alpha\in\Irr(\mathbb{H})$ for short to denote $u^\alpha\in\Irr(\mathbb{H})$. Then $\text{span}\{u^\alpha_{ij}:1\leq i,j\leq n_\alpha,\alpha\in\Irr(\mathbb{H})\}$ is dense in $A$ with respect to the $C^*$-norm \cite[Theorem 5.11]{Compactquantumhypergroup}.
	
	But the orthogonal relation is slightly weaker:
	\begin{thm}\cite[Theorem 5.8, Remark 5.9]{Compactquantumhypergroup}\label{Peter-Weyl quantum hyper group}
		Let $\mathbb{H}=(A,\delta)$ be a compact quantum hypergroup with the Haar measure $h$. Then for irreducible $^\dagger$-representations $[u^\alpha_{ij}]_{i,j=1}^{n_\alpha}$ and $[u^\beta_{kl}]_{k,l=1}^{n_\beta}$, we have
		\[
		h(u^\alpha_{ij}(u^\beta_{lk})^*)=0,
		\]
		if either the representations are not equal or or $i\neq l$.
	\end{thm}  
	
	\subsubsection{Algebraic quantum hypergroups}
	In this subsection we mainly consider *-algebraic quantum hypergroups. Our references in this subsection are \cite{Algebraicquantumhypergroups,Algebraicquantumhypergroups,Discretequantumgroups}.
	
	We start with a *-algebra $A$ over $\mathbb{C}$ with a non-degenerate product. Let $M(A)$ denote its multiplier algebra. A \emph{comultiplication}, or \emph{coproduct} on $A$ is a linear *-preserving map $\Delta:A\to M(A\otimes A)$ such that both $\Delta(a)(1\otimes b)$ and $(a\otimes 1)\Delta(b)$ belong to $A\otimes A$ for all $a,b\in A$ and such that
	\begin{equation*}
	(a\otimes 1\otimes 1)(\Delta\otimes\iota)(\Delta(b)(1\otimes c))=(\iota\otimes\Delta)((a\otimes 1)\Delta(b))(1\otimes 1\otimes c),
	\end{equation*}
	for all $a,b,c\in A$. Note that $\Delta$ is automatically \emph{regular}, i.e., both $\Delta(a)(b\otimes 1)$ and $(1\otimes a)\Delta(b)$ also belong to $A\otimes A$ for all $a,b\in A$.
	
	\begin{defn}\cite{Algebraicquantumhypergroups}
		A \emph{*-algebraic quantum hypergroup} $(A,\Delta,\epsilon,\varphi,S)$ consists of
		\begin{enumerate}
			\item a *-algebra $(A,\Delta)$ with $\Delta$ a comultiplication;
			\item a \emph{counit} $\epsilon:A\to\mathbb{C}$, which is a *-homomorphism such that
			\[
			(\epsilon\otimes\iota)\Delta(a)=a=(\iota\otimes\epsilon)\Delta(a),~~a\in A;
			\]
			\item a \emph{self-adjoint faithful left integral} $\varphi:A\to\mathbb{C}$:
			\[
			\varphi(a^*)=\overline{\varphi(a)},~~a\in A,
			\]
			\[
			(\iota\otimes\varphi)\Delta(a)=\varphi(a)1_{M(A)},~~a\in A,
			\]
			which is \emph{faithful} in the following sense: $a\in A$ must be 0 whenever either $\varphi(ab)=0$ for all $b\in A$ or $\varphi(ba)=0$ for all $b\in A$;
			\item an \emph{antipode $S$ relative to $\varphi$}:
			\[
			S((\iota\otimes\varphi)(\Delta(a)(1\otimes b)))=(\iota\otimes\varphi)((1\otimes a)\Delta(b)),~~a,b\in A.
			\]
		\end{enumerate}  
	\end{defn}
	
	We remark here that the left integrals are unique, under the assumption of the existence of $\varphi$ and the antipode $S$ relative to $\varphi$. Moreover, one can show that the antipode $S$ verifies $S(S(x)^*)^*=x$ for all $x\in A$. In the sequel we use simply $(A,\Delta)$ for short to denote a *-algebraic quantum hypergroup.
	
	Now given a *-algebraic quantum hypergroup $(A,\Delta)$, we explain how to construct its dual $(\hat{A},\hat{\Delta})$, which is again a *-algebraic quantum hypergroup.
	
	We begin with $\hat{A}$. It is defined as the space of the linear functionals on $A$ of the form $\varphi(\cdot a)$ for some $a\in A$. The multiplication $\star$ on $\hat{A}$ is defined in the usual way: $\omega\star\omega':=(\omega\otimes\omega')\hat{\Delta}$ for $\omega,\omega'\in\hat{A}$. For each $\omega\in\hat{A}$, let $\omega^*:=\overline{\omega(S(\cdot)^*)}$. The comultiplication $\hat{\Delta}$ on $\hat{A}$ is given through the formula: $\hat{\Delta}(\omega)(x\otimes y):=\omega(xy)$ for all $\omega\in\hat{A}$ and $x,y\in A$. We should be very careful here since we aim to define $\hat{\Delta}(\omega)$ in $M(\hat{A}\otimes\hat{A})$, instead of $\hat{A}\otimes\hat{A}$. See the discussions before \cite[Proposition 3.7]{Algebraicquantumhypergroups}. The counit $\hat{\epsilon}$ on $\hat{A}$ is given via the formula $\hat{\epsilon}(\varphi(\cdot a))=\varphi(a)$, for all $a\in A$.
	
	To introduce the left integral on $\hat{A}$, we need the \emph{right integral} on $A$. Set $\psi:=\varphi\circ S$, then it is a \emph{right integral}:
	\[
	(\psi\otimes\iota)\Delta(a)=\varphi(a)1_{M(A)},~~a\in A,
	\]
	It can be shown that 
	\[
	\{\varphi(\cdot a):a\in A\}=\{\varphi(a\cdot):a\in A\}=\{\psi(\cdot a):a\in A\}=\{\psi(a\cdot):a\in A\}.
	\]
	Now we can define left integral $\hat{\varphi}$ on $\hat{A}$ as $\hat{\varphi}(\psi(a\cdot)):=\epsilon(a),a\in A$. Finally, the antipode $\hat{S}$ relative to $\hat{\varphi}$ is given by $\hat{S}(\omega):=\omega\circ S$ for all $\omega\in\hat{A}$. Thus $(\hat{A},\hat{\Delta})$ forms also a *-algebraic quantum hypergroup \cite[Theorem 3.11]{Algebraicquantumhypergroups}.
	
	By taking the dual of $(\hat{A},\hat{\Delta})$ again, we recover $(A,\Delta)$. This is the biduality for  *-algebraic quantum hypergroups \cite[Theorem 3.12]{Algebraicquantumhypergroups}. Moreover, a *-algebraic quantum hypergroup is \emph{of compact type} if and only if its dual is \emph{of discrete type}. Here, a *-algebraic quantum hypergroup $(A,\Delta)$ is said to be \emph{of compact type} if $A$ possess an identity 1 (and thus $\Delta(1)=1\otimes1$). And a *-algebraic quantum hypergroup $(A,\Delta)$ is \emph{of discrete type} if it is equipped with a \emph{co-integral} $h\in A$, that is a non-zero element such that $ah=\epsilon(a)h$ for all $a\in A$.
	
	We close this subsection with a question proposed by Delvaux and Van Daele at the end of \cite{Algebraicquantumhypergroups}. Certainly a compact quantum hypergroup is a *-algebraic quantum hypergroup is of compact type. But will a *-algebraic quantum hypergroup of compact type with positive integrals always yield a compact quantum hypergroup?
	
	\subsubsection{Finite quantum hypergroups}
	Our main result lies in the framework of \emph{finite quantum hypergroups}, which are compact quantum hypergroups whose underlying C*-algebras are finite dimensional. The reader should be careful that this notion is different from the one in \cite{Idempotentstatesoncqg}, where one discusses a similar concept introduced only on the algebraic level. 
	
	In a finite quantum hypergroup $\mathbb{H}=(A,\Delta)$, the set of matrix elements of all inequivalent irreducible $^\dagger$-representations $\{u^\alpha_{ij}:1\leq i,j\leq n_\alpha,\alpha\in\Irr(\mathbb{H})\}$ form a basis of $A$. The underlying C*-algebra $\hat{A}$ of its dual (as a *-algebraic quantum hypergroup) $(\hat{A},\hat{\Delta})$, is nothing but $A^*$. Let $\omega^{\alpha}_{ij}$ be the dual basis of $u^{\alpha}_{ij}$ in $A^*$, then from  
	\[
	\Delta(u^\alpha_{ij})=\sum_{k=1}^{n_{\alpha}}u^\alpha_{ik}\otimes u^\alpha_{kj},~~1\leq i,j\leq n_\alpha,\alpha\in\Irr(\mathbb{H}),
	\]
	it follows that
	\[
	\omega^{\alpha}_{ij}\cdot\omega^{\beta}_{kl}(u^{\gamma}_{pq})=\delta_{\alpha\beta}\delta_{jk}\omega^{\alpha}_{il}(u^{\gamma}_{pq}),
	\]
	which yields directly $\omega^{\alpha}_{ij}\cdot\omega^{\beta}_{kl}=\delta_{\alpha\beta}\delta_{jk}\omega^{\alpha}_{il}$. Moreover, 
	\[
	(\omega^{\alpha}_{ij})^*(u^{\beta}_{kl})
	=\overline{\omega^{\alpha}_{ij}(S(u^{\beta}_{kl})^*)}
	=\overline{\omega^{\alpha}_{ij}((u^{\beta}_{kl})^\dagger)}
	=\overline{\omega^{\alpha}_{ij}(u^{\beta}_{lk})}
	=\omega^{\alpha}_{ji}(u^{\beta}_{kl})
	\]
	gives $(\omega^{\alpha}_{ij})^*=\omega^{\alpha}_{ji}$. Hence $\{\omega^\alpha_{ij}:1\leq i,j\leq n_\alpha,\alpha\in\Irr(\mathbb{H})\}$ can be viewed as the matrix units of the dual of $\mathbb{H}$.
	
	\section{Examples of compact quantum hypergroups}
	In this section we present some new constructions of compact quantum hypergroups, which will be of use in the later sections of the paper.
	
	\subsection{Examples from idempotent states}
	
	Suppose that $\mathbb{G}=(A,\Delta)$ is a compact quantum group. Denote by $A^*$ the set of all bounded linear functionals on $A$. For any $\varphi_1,\varphi_2\in A^*$, we can define their \textit{convolution product}, as we did on $\Pol(\mathbb{G})$:
	\[
	\varphi_1\star\varphi_2:=(\varphi_1\otimes\varphi_2)\Delta.
	\]
	We ignore $\star$ if no ambiguity occurs. 
	
	Denote by $\mathcal{S}(A)$ the set of states on $A$. A state $\phi\in\mathcal{S}(A)$ is said to be an \textit{idempotent state} on $A$ if
	\[
	\phi^2=(\phi\otimes\phi)\Delta=\phi.
	\]  
	We use $\text{Idem}\left(\mathbb{G}\right)$ to denote the set of idempotent states on $\mathbb{G}$. Observe first that if $\phi\in\text{Idem}(\mathbb{G})$, we have $\phi=\phi S$ on $\Pol(\mathbb{G})$ \cite{Idempotentstatesoncqg}. In other words, $\hat{\phi}(\alpha)$ is a projection for each $\alpha\in\text{Irr}(\mathbb{G})$; as it is contractive, it is also self-adjoint. 
	
	Now we turn to the study of coidalgebras. Most of the results in this subsection can be found in \cite{preprintofXumin} and \cite{Newcharacterisationofidempotentstates}. A \textit{left (resp. right) coidealgebra} $C$ in a compact quantum group $\mathbb{G}=(A,\Delta)$ is a unital $C^*$-subalgebra in $A$ such that $\Delta(C)\subset C\otimes A$ (resp. $\Delta(C)\subset A\otimes C$).
	
	The first lemma is a special case of \cite[Lemma 3.1]{Newcharacterisationofidempotentstates}, and also a variation of \cite[Lemma 4.3]{CompactquantumgroupsVanDaele}.
	
	\begin{lem}\label{property of idempotent state}
		Let $\phi\in\text{Idem}\left(\mathbb{G}\right)$. For $b\in A$ define $\phi_b\left(a\right):=\phi\left(ab\right)$ for all $a\in A$. Then we have
		\[
		\phi\star\phi_b=\phi\left(b\right)\phi.
		\]
	\end{lem}
	
	For $\phi\in\text{Idem}\left(\mathbb{G}\right)$ set $\mathbb{E}_\phi^\ell:=\left(\phi\otimes\iota\right)\Delta$ and $\mathbb{E}_\phi^r:=\left(\iota\otimes\phi\right)\Delta$. The next lemma lists some useful properties of these maps.
	
	\begin{lem}\label{property of E^l&E^r}
		Let $\mathbb{G}=(A,\Delta)$ be a compact quantum group. Then
		\begin{enumerate}
			\item 
			$\mathbb{E}_\phi^\ell\left(a^*\right)=\mathbb{E}_\phi^\ell\left(a\right)^*,\mathbb{E}_\phi^r\left(a^*\right)=\mathbb{E}_\phi^r\left(a\right)^*,a\in A$.
			\item $\Delta\mathbb{E}_\phi^\ell=(\mathbb{E}_\phi^\ell\otimes\iota)\Delta,\Delta\mathbb{E}_\phi^r=(\iota\otimes\mathbb{E}_\phi^r)\Delta$.
			\item $(\iota\otimes\mathbb{E}_\phi^\ell)\Delta=(\mathbb{E}_\phi^r\otimes\iota)\Delta$.
			\item $\mathbb{E}_\phi^\ell\mathbb{E}_\phi^r=\mathbb{E}_\phi^r\mathbb{E}_\phi^\ell$.
			\item $\mathbb{E}_\phi^\ell(\mathbb{E}_\phi^\ell\left(a\right)\mathbb{E}_\phi^\ell\left(b\right))=\mathbb{E}_\phi^\ell\left(a\right)\mathbb{E}_\phi^\ell\left(b\right),
			\mathbb{E}_\phi^r(\mathbb{E}_\phi^r\left(a\right)\mathbb{E}_\phi^r\left(b\right))=\mathbb{E}_\phi^r\left(a\right)\mathbb{E}_\phi^r\left(b\right),a,b\in A$. Consequently, $\mathbb{E}_\phi^\ell\mathbb{E}_\phi^ \ell=\mathbb{E}_\phi^\ell$, $\mathbb{E}_\phi^r\mathbb{E}_\phi^r=\mathbb{E}_\phi^r$.
		\end{enumerate}
		Hence $\mathbb{E}_\phi^\ell\left(A\right)$ is a left coidalgebra and $\mathbb{E}_\phi^r\left(A\right)$ is a right coidalgebra of $A$.
	\end{lem}
	
	\begin{proof}
		(1)-(4) are just straightforward computations. For (5) we prove the statement only for $\mathbb{E}_\phi^\ell$, as the proof for $\mathbb{E}_\phi^r$ is similar. For this note that it suffices to show the first equation for any $a,b$ the coefficients of unitary representation of $\mathbb{G}$. The case for general $a,b$ follows from the density argument. Let $\left\{u^\alpha_{ij},1\leq i,j\leq n_\alpha\right\}$ be the coefficients of the irreducible representation $u^\alpha,\alpha\in\Irr\left(\mathbb{G}\right)$. Since $\phi\star\phi=\phi$, we have
		\begin{equation}\label{Equation1}
		\phi\left(u^\alpha_{ij}\right)=\sum_{k=1}^{n_\alpha}\phi\left(u^\alpha_{ik}\right)\phi\left(u^\alpha_{kj}\right),~~1\leq i,j\leq n_\alpha,\alpha\in\Irr\left(\mathbb{G}\right).
		\end{equation}
		From Lemma \ref{property of idempotent state} we have for any $b\in A$
		\begin{equation}\label{Equation2}
		\phi\left(b\right)\phi\left(u^\alpha_{ij}\right)=\sum_{k=1}^{n_\alpha}\phi\left(u^\alpha_{ik}\right)\phi\left(u^\alpha_{kj}b\right),~~1\leq i,j\leq n_\alpha,\alpha\in\Irr\left(\mathbb{G}\right).
		\end{equation}
		So for any $u^\alpha_{ij}$ and $u^\beta_{kl}$ we have
		\[
		\mathbb{E}_\phi^\ell(u^\alpha_{ij})\mathbb{E}_\phi^\ell(u^\beta_{kl})
		=\sum_{s=1}^{n_\alpha}\sum_{t=1}^{n_\beta}\phi(u^\alpha_{is})\phi(u^\beta_{kt})u^\alpha_{sj}u^\beta_{tl},
		\]
		and 
		\begin{align*}
		&\mathbb{E}_\phi^\ell\left(\mathbb{E}_\phi^\ell\left(u^\alpha_{ij}\right)\mathbb{E}_\phi^\ell(u^\beta_{kl})\right)\\
		=&\sum_{s=1}^{n_\alpha}\sum_{t=1}^{n_\beta}\phi(u^\alpha_{is})\phi(u^\beta_{kt})\left(\phi\otimes\iota\right)\left(\sum_{p=1}^{n_\alpha}\sum_{q=1}^{n_\beta}u^\alpha_{sp}u^\beta_{tq}\otimes u^\alpha_{pj}u^\beta_{ql}\right)\\
		=&\sum_{s,p=1}^{n_\alpha}\sum_{t,q=1}^{n_\beta}\phi(u^\alpha_{is})\phi(u^\beta_{kt})\phi(u^\alpha_{sp}u^\beta_{tq})u^\alpha_{pj}u^\beta_{ql}\\
		=&\sum_{p=1}^{n_\alpha}\sum_{t,q=1}^{n_\beta}\phi(u^\beta_{kt})\left(\sum_{s=1}^{n_\alpha}\phi(u^\alpha_{is})\phi(u^\alpha_{sp}u^\beta_{tq})\right)u^\alpha_{pj}u^\beta_{ql}.
		\end{align*}	
		Now apply \eqref{Equation2} for $b=u^\beta_{tq}$ and use \eqref{Equation1}, we get
		\begin{align*}
		\mathbb{E}_\phi^\ell\left(\mathbb{E}_\phi^\ell\left(u^\alpha_{ij}\right)\mathbb{E}_\phi^\ell(u^\beta_{kl})\right)
		=&\sum_{p=1}^{n_\alpha}\sum_{t,q=1}^{n_\beta}\phi(u^\beta_{kt})\phi(u^\alpha_{ip})\phi(u^\beta_{tq})u^\alpha_{pj}u^\beta_{ql}\\
		=&\sum_{p=1}^{n_\alpha}\sum_{q=1}^{n_\beta}\phi(u^\alpha_{ip})\left(\sum_{t=1}^{n_\beta}\phi(u^\beta_{kt})\phi(u^\beta_{tq})\right)u^\alpha_{pj}u^\beta_{ql}\\
		=&\sum_{p=1}^{n_\alpha}\sum_{q=1}^{n_\beta}\phi(u^\alpha_{ip})\phi(u^\beta_{kq})u^\alpha_{pj}u^\beta_{ql}.
		\end{align*}
		Hence $\mathbb{E}_\phi^\ell\left(\mathbb{E}_\phi^\ell\left(u^\alpha_{ij}\right)\mathbb{E}_\phi^\ell(u^\beta_{kl})\right)=\mathbb{E}_\phi^\ell\left(u^\alpha_{ij}\right)\mathbb{E}_\phi^\ell(u^\beta_{kl})$. The remaining part is a consequence of the equality  $\mathbb{E}_\phi^\ell(1)=1$.
	\end{proof}
	
	Now we define $\mathbb{E}_\phi:=\mathbb{E}_\phi^\ell\mathbb{E}_\phi^r=\mathbb{E}_\phi^r\mathbb{E}_\phi^\ell$. The map $\mathbb{E}_\phi$ shares the similar properties of $\mathbb{E}_\phi^\ell$ and $\mathbb{E}_\phi^r$, see the following lemma (1)-(4). Moreover, $\mathbb{E}_\phi$ commutes with the antipode $S$, as the following lemma (5) shows, which means that the algebra $A_{\phi}:=\mathbb{E}_\phi(A)$ possesses nicer properties than $\mathbb{E}_\phi^\ell(A)$ and $\mathbb{E}_\phi^r(A)$. That is why we use the map $\mathbb{E}_\phi$ in the remaining part of the paper.
	
	\begin{lem}\label{property of E}
		Let $\mathbb{G}=(A,\Delta)$ be a compact quantum group. Then
		\begin{enumerate}
			\item $\mathbb{E}_\phi\left(a^*\right)=\mathbb{E}_\phi\left(a\right)^*,a\in A$. 
			\item $\Delta\mathbb{E}_\phi=(\mathbb{E}_\phi^\ell\otimes\mathbb{E}_\phi^r)\Delta$.
			\item $(\mathbb{E}_\phi\otimes\mathbb{E}_\phi)\Delta=(\iota\otimes\mathbb{E}_\phi^\ell)\Delta\mathbb{E}_\phi=(\mathbb{E}_\phi^r\otimes\iota)\Delta\mathbb{E}_\phi$.
			\item 
			$\mathbb{E}_\phi\left(\mathbb{E}_\phi\left(a\right)\mathbb{E}_\phi\left(b\right)\right)=\mathbb{E}_\phi\left(a\right)\mathbb{E}_\phi\left(b\right),a,b\in A$. Consequently, $\mathbb{E}_\phi\mathbb{E}_\phi=\mathbb{E}_\phi$.
		\end{enumerate}
		Hence $A_{\phi}$ is a unital $C^*$-subalgebra of $A$. Moreover, 
		\begin{enumerate}
			\item [(5)] $S\mathbb{E}_\phi=\mathbb{E}_\phi S$ on $\Pol(\mathbb{G})$.
		\end{enumerate}
	\end{lem}
	
	\begin{proof}
		Again we omit the proof of (1)-(4) here. The fact that $A_{\phi}$ is a unital $C^*$-subalgebra follows directly from these properties. To prove (5), it suffices to check the equality for $u^\alpha_{ij}$, the coefficients of unitary representation of $\mathbb{G}$, where $1\leq i,j\leq n_\alpha,\alpha\in\text{Irr}(\mathbb{G})$. And that is a consequence of 
		\begin{align*}
		S\mathbb{E}_\phi(u^\alpha_{ij})
		&=S(\phi\otimes\iota\otimes\phi)\Delta^{(2)}(u^\alpha_{ij})\\
		&=\sum_{k,l=1}^{n_\alpha}\phi(u^\alpha_{ik})\phi(u^\alpha_{lj})S(u^\alpha_{kl})\\
		&=\sum_{k,l=1}^{n_\alpha}\phi(u^\alpha_{ik})\phi(u^\alpha_{lj})(u^\alpha_{lk})^*,
		\end{align*}
		and 
		\begin{align*}
		\mathbb{E}_\phi S(u^\alpha_{ij})
		&=(\phi\otimes\iota\otimes\phi)\Delta^{(2)}((u^\alpha_{ji})^*)\\
		&=\sum_{k,l=1}^{n_\alpha}\phi((u^\alpha_{ki})^*)\phi((u^\alpha_{jl})^*)(u^\alpha_{lk})^*\\
		&=\sum_{k,l=1}^{n_\alpha}\phi(u^\alpha_{ik})\phi(u^\alpha_{lj})(u^\alpha_{lk})^*,
		\end{align*}
		where the last equality follows from the facts that $\phi=\phi S$ on $\Pol(\mathbb{G})$ and $S(u^\alpha_{lj})=(u^\alpha_{jl})^*$.
	\end{proof}
	
	The following proposition says that  $\mathbb{E}_\phi$ as above  is a projection verifying all the conditions in Theorem \ref{New examples}. Thus $(A_{\phi},\Delta_{\phi})$ becomes a compact quantum hypergroup. 
	
	\begin{prop}
		Let $\mathbb{G}=(A,\Delta,\epsilon,S)$ be a compact quantum group. Let $\mathbb{E}_\phi$ be defined as above and set $\Delta_{\phi}:=(\mathbb{E}_\phi\otimes\mathbb{E}_\phi)\Delta|_{A_{\phi}}$. Then $\mathbb{E}_\phi$ is an $h$-invariant conditional expectation such that
		\begin{enumerate}
			\item $(\mathbb{E}_\phi\otimes \mathbb{E}_\phi)\Delta(x)=(\mathbb{E}_\phi\otimes \mathbb{E}_\phi)\Delta(\mathbb{E}_\phi(x)),x\in A$;
			\item $\Pol(\mathbb{G})$ is invariant under $\mathbb{E}_\phi$ and $S \mathbb{E}_\phi=\mathbb{E}_\phi S$; 
			\item the restriction of $\epsilon$ to $A_{\phi}$ is a counit, i.e.,
			\[
			(\epsilon\otimes\iota)\Delta_{\phi}=\iota=(\iota\otimes\epsilon)\Delta_{\phi}.
			\]
		\end{enumerate}
		Hence by Theorem \ref{New examples}, $(A_{\phi},\Delta_{\phi})$ is a compact quantum hypergroup.
	\end{prop}
	
	\begin{proof}
		The proof is now based on straightforward computations using the last two lemmas.
	\end{proof}
	
	Let $\phi$ be an idempotent state on compact quantum group $\mathbb{G}=(A,\Delta)$. A functional $u\in A^*$ is called \textit{$\phi$-bi-invariant} if $u\phi=\phi u=u$.
	
	In the remaining part of this subsection we characterize the $\phi$-bi-invariant functionals, where $\phi$ is an idempotent state.  It turns out that one can transfer each $\phi$-bi-invariant functional on $A$ to its restriction to $A_{\phi}$, preserving the norm and the *-algebra structure. See also \cite{preprintofXumin} for related work.
	
	For linear functionals $\varphi_1,\varphi_2$ on compact quantum hypergroup $(A_{\phi},\Delta_{\phi})$ with the antipode $S$ in $(A,\Delta)$, one can also define the convolution and the involution as we did on compact quantum groups (the notations here are slightly different from the ones in Section 1):
	\[
	\varphi_1\star \varphi_2:=(\varphi_1\otimes \varphi_2)\Delta_{\phi},\ \ \varphi_1^*:=\overline{\varphi_1(S(\cdot)^*)}.
	\]
	Still, we write $\varphi_1\varphi_2$ for short to denote $\varphi_1\star \varphi_2$. Note here that $\varphi_1^*$ is well-defined because $S\mathbb{E}_\phi=\mathbb{E}_\phi S$.
	
	We formulate the results of $\phi$-bi-invariant functionals here without the proof.
	\begin{lem}\label{bi-invariance}
		Let $\phi\in\text{Idem}(\mathbb{G})$ and $u\in A^*$. Then $u$ is $\phi$-bi-invariant if and only if $u=u|_{A_{\phi}}\mathbb{E}_\phi$. In this case, the following hold:
		\begin{enumerate}
			\item $\Vert u\Vert=\Vert u|_{A_{\phi}}\Vert$;
			\item $u$ is positive (resp. a state) if and only if $u|_{A_{\phi}}$ is positive (resp. a state);
			\item $u^*|_{A_{\phi}}=(u|_{A_{\phi}})^*$;
			\item if $u$ and $v$ in $A^*$ are both $\phi$-bi-invariant, then $(uv)|_{A_{\phi}}=u|_{A_{\phi}}v|_{A_{\phi}}$;
			\item $\phi=\epsilon\mathbb{E}_\phi=\epsilon|_{A_{\phi}}\mathbb{E}_\phi$.
		\end{enumerate} 
	\end{lem}
	
	\begin{rem}
		The last property implies that the idempotent state $\phi$ can be recovered by the counit $\epsilon$ through this formula. This will be frequently used in the sequel.
	\end{rem}
	
	\subsection{Examples from group-like projections}
	
	Group-like projections in algebraic quantum groups were first introduced by Van Daele and Landstad in \cite{grouplikeprojections}. The relation between idempotent states and group-like projections has been studied by Franz and  Skalski in \cite{Idempotentstatesoncqg} on compact quantum groups, and then by Faal and Kasprzak in \cite{Group-likeprojectionsforLCQG}, and Kasparzak and So\l tan in \cite{LatticeofidempotentstatesonLCQG} on locally compact quantum groups. 
	
	The main result in this subsection is that a \emph{group-like projection} in a compact quantum group induces a compact quantum hypergroup. The main ingredients were obtained on the algebraic quantum group level by Delvaux and Van Daele \cite{constructionsandexamplesofquantumhypergroups}.
	
	\begin{defn}
		Let $\mathbb{G}=(A,\Delta)$ be a compact quantum group. A non-zero element $p\in A$ is called a \emph{group-like projection} if $p=p^*=p^2$ and 
		\begin{equation}\label{group-like projection}
		\Delta(p)(1\otimes p)=p\otimes p=\Delta(p)(p\otimes 1),~~S(p)=p.   
		\end{equation}
		Note that by taking adjoints we have
		\begin{equation}\label{group-like projection-bis}
		(1\otimes p)\Delta(p)=p\otimes p=(p\otimes 1)\Delta(p).   
		\end{equation}
	\end{defn}
	
	The following proposition is not a direct consequence of Theorem \ref{New examples}, since the projection $P$ is not $h$-invariant in general. But one can check the proof \cite[Theorem 2.1]{Newexamples} to see that this is only used to deduce the strong invariance \eqref{strong invariance} of $h$, which is nothing but (1) of the Proposition \ref{examples from group-like projections}.
	
	\begin{prop}\label{examples from group-like projections}
		Let $(A,\Delta)$ be a compact quantum group. Let $p$ be a group-like projection in $A$, then $P:A\to A_p:=pAp,a\mapsto pap$ is an $h$-preserving conditional expectation such that
		\begin{enumerate}
			\item $(\iota\otimes h)((S\otimes\iota)\Delta_p(pap)(1\otimes pbp))=(\iota\otimes h)((1\otimes pap)\Delta_p(pbp)),~~a,b\in A$;
			\item $(P\otimes P)\Delta=(P\otimes P)\Delta P$;
			\item $S P=P S$; 
			\item the restriction of $\epsilon$ to $A_p$ is a counit, i.e.,
			\[
			(\epsilon\otimes\iota)\Delta_p=\iota=(\iota\otimes\epsilon)\Delta_p,
			\]
		\end{enumerate}
		where $\Delta_p:=(P\otimes P)\Delta|_{A_p}$. Then $(A_p,\Delta_p)$ is a compact quantum hypergroup.
	\end{prop}
	
	\begin{proof}
		\begin{enumerate}
			\item For any $a,b\in A$, we have by \eqref{group-like projection}
			\begin{align*}
			(\iota\otimes h)((S\otimes\iota)\Delta_p(pap)(1\otimes pbp))
			&=(\iota\otimes h)(S\otimes\iota)(\Delta_p(pap)(1\otimes pbp))\\
			&=(\iota\otimes h)(S\otimes\iota)((p\otimes p)\Delta(pap)(1\otimes pbp)(p\otimes p))\\
			&=(\iota\otimes h)(S\otimes\iota)((p\otimes 1)\Delta(pap)(1\otimes pbp)(p\otimes 1))\\
			&=p(\iota\otimes h)(S\otimes\iota)(\Delta(pap)(1\otimes pbp))p\\
			&=p(\iota\otimes h)((1\otimes pap)\Delta(pbp))p\\
			&=(\iota\otimes h)((p\otimes 1)(1\otimes pap)\Delta(pbp)(p\otimes 1))\\
			&=(\iota\otimes h)((1\otimes pap)(p\otimes p)\Delta(pbp)(p\otimes p))\\
			&=(\iota\otimes h)((1\otimes pap)\Delta_p(pbp)).
			\end{align*}
			
			\item For any $a\in A$, we have by \eqref{group-like projection} and \eqref{group-like projection-bis} 
			\begin{align*}
			(P\otimes P)\Delta(a)
			&=(p\otimes p)\Delta(a)(p\otimes p)\\
			&=(p\otimes p)(p\otimes p)\Delta(a)(p\otimes p)(p\otimes p)\\
			&=(p\otimes p)(1\otimes p)\Delta(p)\Delta(a)\Delta(p)(p\otimes 1)(p\otimes p)\\
			&=(p\otimes p)\Delta(pap)(p\otimes p)\\
			&=(P\otimes P)\Delta P(a).
			\end{align*}
			
			\item For any $a\in A$, it follows from the definition of group-like projection that
			\[
			S P(a)=S(pap)=S(p)S(a)S(p)=pS(a)p=P S(a).
			\]
			
			\item By definitions of the counit $\epsilon$ on $A$ and the group-like projection $p$ we obtain
			\begin{align*}
			(\epsilon\otimes\iota)\Delta_p(pap)
			&=(\epsilon\otimes\iota)((p\otimes p)\Delta(pap)(p\otimes p))\\
			&=(\epsilon\otimes\iota)((p\otimes p)\Delta(p)\Delta(a)\Delta(p)(p\otimes p))\\
			&=(\epsilon\otimes\iota)((p\otimes p)\Delta(a)(p\otimes p))\\ 
			&=\epsilon(p)^2 p(\epsilon\otimes\iota)\Delta(a)p\\
			&=pap,
			\end{align*}
			where we have used un easy fact $\epsilon(p)=1$ in the last equality. So $(\epsilon\otimes\iota)\Delta_p=\iota$ and the proof of  the other equality is similar.
		\end{enumerate}
	\end{proof}
	
	\subsection{A duality theorem}\label{Duality induced by an idempotent state}
	Let $(A,\Delta)$ be a finite quantum group. Let $\phi$ be an idempotent state on $A$. Denote by $(A_{\phi},\Delta_{\phi})$ the finite quantum hypergroup induced by $\phi$. Considered as an element in $\hat{A}$, $p=\phi$ is a group-like projection \cite{Idempotentstatesoncqg}. We then let $(\hat{A}_{p},\hat{\Delta}_{p})$ be the finite quantum hypergroup associated to the group-like projection $p=\phi$ in $\hat{A}$. We show that $(\hat{A}_{p},\hat{\Delta}_{p})$ is the dual of $(A_{\phi},\Delta_{\phi})$.
	
	\begin{thm}
		Let $\phi$ be an idempotent state on a finite quantum group $(A,\Delta)$. Let $(A_{\phi},\Delta_{\phi})$ and $(\hat{A}_{p},\hat{\Delta}_{p})$ be as above. Then $(\hat{A}_{p},\hat{\Delta}_{p})$ is isomorphic to the dual of $(A_{\phi},\Delta_{\phi})$.
	\end{thm}
	
	\begin{proof}
		Indeed, the key ingredients of the proof have already been included in Lemma \ref{bi-invariance}. Let $\pi(\varphi):=\varphi\mathbb{E}_{\phi}$ for each $\varphi\in\widehat{A_{\phi}}=(A_{\phi})^*$. Then by Lemma \ref{bi-invariance}, $\pi$ is a bijection between $\widehat{A_{\phi}}$ and $\hat{A}_{p}$. For any $\varphi_1,\varphi_2\in\widehat{A_{\phi}}=(A_{\phi})^*$, we have by Lemma \ref{bi-invariance} that 
		\[
		\pi(\varphi_1\varphi_2)
		=(\varphi_1\otimes\varphi_2)\Delta_{\phi}\mathbb{E}_{\phi}
		=(\varphi_1\mathbb{E}_{\phi}\otimes\varphi_2\mathbb{E}_{\phi})\Delta_{\phi}
		=(\pi(\varphi_1)\otimes\pi(\varphi_2))\Delta_{\phi}
		=\pi(\varphi_1)\pi(\varphi_2).
		\]
		Let $\varphi\in (A_{\phi})^*$, then by Lemma \ref{property of E}
		\[
		\pi(\varphi^*)(a)
		=\overline{\varphi(S(\mathbb{E}_{\phi}(a))^*)}
		=\overline{\varphi(\mathbb{E}_{\phi}(S(a))^*)}
		=\overline{\varphi\mathbb{E}_{\phi}(S(a)^*)}
		=\pi(\varphi)^*(a),
		\]
		for all $a\in A$. So $\pi$ is a *-isomorphism. It remains to show that $(\pi\otimes\pi)\widehat{\Delta_{\phi}}=\hat{\Delta}_{p}\pi$. To see this, let $\varphi\in(A_{\phi})^*$ and let $a,b\in A$. Then from Lemma \ref{property of E} and Lemma \ref{bi-invariance} it follows
		\[
		\hat{\Delta}_{p}\pi(\varphi)(a\otimes b)
		=\hat{\Delta}_{p}(\varphi\mathbb{E}_{\phi})(a\otimes b)
		=(\varphi\mathbb{E}_{\phi})(ab)
		=(\varphi\mathbb{E}_{\phi})(\mathbb{E}_{\phi}(a)\mathbb{E}_{\phi}(b))
		=\varphi(\mathbb{E}_{\phi}(a)\mathbb{E}_{\phi}(b)),
		\]
		which is nothing but $(\pi\otimes\pi)\widehat{\Delta_{\phi}}(a\otimes b)$. This finishes the proof.
	\end{proof}
	
	\section{Poisson states on compact quantum groups} 
	
	Let $\mathbb{G}=(A,\Delta)$ be a compact quantum group. For each $\phi\in\text{Idem}(\mathbb{G})$, we say that $\{\omega_t\}_{t\geq0}$ is a \textit{convolution semigroup of functionals} on $A$ starting from $\phi$ if 
	\begin{enumerate}
		\item $\omega_t\in A^*$ for each $t\geq 0$.
		\item $\omega_{s+t}=\omega_{s}\omega_{t}$ for all $s,t\geq0$.
		\item $\omega_0=\phi$.
	\end{enumerate}
	If moreover, each $\omega_t\in\mathcal{S}(A)$, we call $\{\omega_t\}_{t\geq0}$ a \textit{convolution semigroup of states} starting from $\phi$. We say that the convolution semigroup of states $\{\omega_t\}_{t\geq0}$ is \textit{norm continuous} if 
	\[
	\lim\limits_{t\to 0^+}\Vert\omega_t-\phi\Vert=0.
	\]
	
	For a $\phi$-bi-invariant functional $u\in A^*$  define
	\[
	\exp_\phi(u):=\phi+\sum_{n\geq 1}\frac{u^n}{n!}.
	\]  
	Then it is easy to check that $\{\exp_\phi(tu)\}_{t\geq0}$ form a norm continuous convolution semigroup of functionals. We aim to find sufficient and necessary conditions on $u$ such that $\{\exp_\phi(tu)\}_{t\geq0}$ is a convolution semigroup of states. For this we make some notations. A functional $u\in A^*$ is called $\textit{Hermitian}$ if $u\left(x^*\right)=\overline{u\left(x\right)}$ for all $x$; it is further called  \textit{conditionally positive definite with respect to $\phi$} if $u\left(x^*x\right)\geq0$ for all $x$ such that $\phi\left(x^*x\right)=0$.
	The main theorem in this section is as follows.
	
	\begin{thm}\label{Theorem 1}
		Suppose that $\mathbb{G}=\left(A,\Delta\right)$ is a compact quantum group. Let $\phi\in\text{Idem}\left(\mathbb{G}\right)$. Then for $u\in A^*$, the following statements are equivalent.
		\begin{enumerate}
			\item $u(1)=0$, and $u$ is conditionally positive definite with respect to $\phi$.
			\item $u=r(v-\phi)$, where $r\geq0$ and $v$ is a $\phi$-bi-invariant state.
		\end{enumerate}
	\end{thm}
	
	The following proposition proves Theorem \ref{Theorem 1} on general unital $C^*$-algebras, under the additional assumptiion that  $\phi=\varepsilon$ is a character (the quantum group structure can be then removed, since any $u\in A^*$ is $\epsilon$-bi-invariant with $\epsilon$ the counit).
	
	\begin{prop}\label{counit case}
		Let $A$ be a unital $C^*$-algebra with $\varepsilon$ a character. Then for any non-zero bounded linear functional $u$ on $A$ such that $u(1)=0$ and $u(x^*x)\geq0$ for all $\varepsilon(x^*x)=0$, we have $u=r(v-\varepsilon)$, where $r>0$ and $v$ is a state.
	\end{prop}
	
	\begin{proof}
		Note first that $\varepsilon(x^*x)=|\varepsilon(x)|^2$. So $\varepsilon(x^*x)=0$ if and only if $x\in\ker\varepsilon=\{x:\varepsilon(x)=0\}$. Let $u_0:=u|_{\ker\varepsilon}$ be the restriction of $u$ to $\ker\varepsilon$. By assumption, $u_0$ is a bounded linear positive functional on the ideal $\ker\varepsilon$. So it admits a unique positive linear extension $\widetilde{u_0}$ to $A$ such that $\widetilde{u_0}|_{\ker\varepsilon}=u_0$ and $\Vert\widetilde{u_0}\Vert=\Vert u_0\Vert$. Hence for any $x\in A$, we have $x-\varepsilon(x)1\in\ker\varepsilon$ and thus 
		\[
		u(x)=u\left(x-\varepsilon(x)1\right)=u_0\left(x-\varepsilon(x)1\right)=\widetilde{u_0}\left(x-\varepsilon(x)1\right)=r(v-\varepsilon)(x),
		\]
		where $r:=\Vert\widetilde{u_0}\Vert=\Vert u_0\Vert>0$ and $v:=\frac{1}{r}\widetilde{u_0}$ is a state.
	\end{proof}
	
	Now we are ready to prove Theorem \ref{Theorem 1}. The idea is to restrict the problem to $A_{\phi}$, and then apply Proposition \ref{counit case} by recovering the idempotent state $\phi$ as $\phi=\epsilon|_{A_{\phi}}\mathbb{E}_\phi$, where $\epsilon|_{A_{\phi}}$ is a character on $A_{\phi}$.
	
	\begin{proof}[Proof of the Theorem \ref{Theorem 1}]
		The direction $(1)\Rightarrow(2)$ is clear. To prove $(2)\Rightarrow(1)$, suppose $u\neq0$ and write $u=u|_{A_{\phi}}\mathbb{E}_\phi$ by Lemma \ref{bi-invariance}. Note that $\epsilon|_{A_{\phi}}$ is a character on the unital $C^*$-algebra $A_{\phi}$. From the definition of $u$, we have $u|_{A_{\phi}}(1)=0$. Moreover, for any $x\in A$ such that $\epsilon|_{A_{\phi}}\left(\mathbb{E}_\phi\left(x\right)^*\mathbb{E}_\phi\left(x\right)\right)=0$, we have by Lemma \ref{property of E^l&E^r} and Lemma \ref{bi-invariance} that
		\[
		0
		=\epsilon|_{A_{\phi}}\left(\mathbb{E}_\phi\left(x\right)^*\mathbb{E}_\phi\left(x\right)\right)
		=\epsilon|_{A_{\phi}}\mathbb{E}_\phi\left(\mathbb{E}_\phi\left(x\right)^*\mathbb{E}_\phi\left(x\right)\right)
		=\phi\left(\mathbb{E}_\phi\left(x\right)^*\mathbb{E}_\phi\left(x\right)\right).
		\]
		Again, by Lemma \ref{property of E^l&E^r} and Lemma \ref{bi-invariance}, the conditionally positive definiteness of $u$ with respect to $\phi$ implies \[u|_{A_{\phi}}\left(\mathbb{E}_\phi\left(x\right)^*\mathbb{E}_\phi\left(x\right)\right)
		=u|_{A_{\phi}}\mathbb{E}_\phi\left(\mathbb{E}_\phi\left(x\right)^*\mathbb{E}_\phi\left(x\right)\right)
		=u\left(\mathbb{E}_\phi\left(x\right)^*\mathbb{E}_\phi\left(x\right)\right)\geq0.
		\]
		So we have by Proposition \ref{counit case} that  $u|_{A_{\phi}}=r(w-\epsilon|_{A_{\phi}})$ with $r>0$ and $w$ a state on $A_{\phi}$. Set $v:=w\mathbb{E}_\phi$, then $v$ is, by Lemma \ref{bi-invariance}, a $\phi$-bi-invariant state on $A$ such that 
		\[
		u
		=u|_{A_{\phi}}\mathbb{E}_\phi
		=r\left(w\mathbb{E}_\phi-\epsilon|_{A_{\phi}}\mathbb{E}_\phi\right)
		=r(v-\phi),
		\]
		as desired.
	\end{proof} 
	
	\begin{defn}
		Let $\phi\in\text{Idem}(\mathbb{G})$. We denote by $\mathcal{N}_\phi(\mathbb{G})$ the class of $u\in A^*$ that satisfy the conditions in Theorem \ref{Theorem 1}. By condition (2), $\omega=\exp_\phi(u)$ is a state for each $u\in \mathcal{N}_\phi(\mathbb{G})$. Denote by $\mathcal{P}_\phi(\mathbb{G})$ the set of all such states. Set $\mathcal{P}(\mathbb{G}):=\bigcup_{\phi\in\text{Idem}(\mathbb{G})}\mathcal{P}_\phi(\mathbb{G})$. Then any $\omega\in\mathcal{P}(\mathbb{G})$ is said to be of \emph{Poisson type}, or a \textit{Poisson state} on $\mathbb{G}$. 
	\end{defn}
	
	Recall that any norm continuous convolution semigroup of states $\{\omega_t\}_{t\geq0}$ on a compact quantum group $\mathbb{G}=(A,\Delta)$ can be recovered by exponentiation with the bounded generator $u:=\lim_{t\to0^+}\frac{1}{t}(\omega_t-\omega_0)$. It is not difficult to see that $u(1)=0$ and $u$ is conditionally positive definite with respect to $\omega_0$, since these hold for each $\frac{1}{t}(\omega_t-\omega_0),t>0$. Then together with Theorem \ref{Theorem 1} we have the following result. 
	
	\begin{thm}
		Let $\phi$ be an idempotent state on a compact quantum group $\mathbb{G}=(A,\Delta)$. For any non-zero bounded linear functional $\omega$ on $A$ such that $\omega\phi=\phi\omega=\omega$, the following are equivalent
		\begin{enumerate}
			\item $\omega=\omega_1$ with $\{\omega_t\}_{t\geq0}$ a norm continuous convolution semigroup of states such that $\omega_0=\phi$;
			\item $\omega=\exp_\phi(u)$, where $u \in A^*$ is $\phi$-bi-invariant,  $u(1)=0$, and $u(x^*x)\geq 0$ for all $x\in A$ such that $\phi(x^*x)=0$;
			\item $\omega=\exp_\phi(u)$, where $u=r(v-\phi)$, with $r>0$ and $v$ a $\phi$-bi-invariant state on $A$.
		\end{enumerate}
	\end{thm}
	
	\section{Infinitely divisible states on finite quantum groups}
	In this section we prove the main result of the paper.
	
	\begin{defn}
		Let $\mathbb{G}=(A,\Delta)$ be a compact quantum group. A state $\omega\in\mathcal{S}(A)$ is said to be \textit{infinitely divisible} if $\omega=\omega_n^n$ for some $\omega_n\in\mathcal{S}(A)$ for all $n\geq 1$. We use $\mathcal{I}(\mathbb{G})$ to denote the set of all infinitely divisible states on $\mathbb{G}$.
	\end{defn}
	
	Clearly Poisson states are infinitely divisible. Our main result in this section is that any infinitely divisible state on a finite quantum group is a Poisson state. From now on, unless stated otherwise, $\mathbb{G}=(A,\Delta)$ always denotes a finite quantum group. 
	
	The following lemma is well-known, and the proof follows from standard arguments.
	
	\begin{lem}\label{characterization of positive functional}
		Let $B=\oplus_{k=1}^{m}\mathbb{M}_{n_k}(\mathbb{C})$ with matrix units $\{e^k_{ij}:1\leq i,j\leq n_k,1\leq k\leq m\}$. Denote the dual basis by $\{\omega^k_{ij}\}$. Then for any $\omega=\sum_{k=1}^{m}\sum_{i,j=1}^{n_k}a^{(k)}_{ij}\omega^{k}_{ij}$, $\omega$ is a positive linear functional if and only if $[a^{(k)}_{ij}]_{i,j=1}^{n_k}$ is positive semi-definite for each $k$. In this case, $\Vert\omega\Vert=\sum_{k=1}^{m}\sum_{i=1}^{n_k}a^{(k)}_{ii}$.
	\end{lem}
	
	As a direct consequence, we have the following Jordan type decomposition, which is quite easy but very helpful.
	
	\begin{cor}\label{Jordan decomposition}
		Let $B=\oplus_{k=1}^{m}\mathbb{M}_{n_k}(\mathbb{C})$ with matrix units $\{e^k_{ij}:1\leq i,j\leq n_k,1\leq k\leq m\}$. Denote the dual basis by $\{\omega^k_{ij}\}$. Let $\omega=\sum_{k=1}^{m}\sum_{i,j=1}^{n_k}a^{(k)}_{ij}\omega^{k}_{ij}$ such that either $[a^{(k)}_{ij}]_{i,j=1}^{n_k}\geq0$ or $[a^{(k)}_{ij}]_{i,j=1}^{n_k}\leq0$. Then $\omega_+:=\sum_{k\in\Lambda}\sum_{i,j=1}^{n_k}a^{(k)}_{ij}\omega^{k}_{ij}$ and $\omega_-:=\sum_{k\notin\Lambda}\sum_{i,j=1}^{n_k}a^{(k)}_{ij}\omega^{k}_{ij}$ are positive functionals on $B$ such that
		\[
		\omega=\omega_+ -\omega_- \text{ and }\Vert\omega\Vert=\Vert\omega_+\Vert+\Vert\omega_-\Vert,
		\]
		where $\Lambda$ is the set of all the $k$'s such that $[a^{(k)}_{ij}]_{i,j=1}^{n_k}\geq0$.
	\end{cor}
	
	Another important corollary is as follows.
	
	\begin{cor}\label{Corollay about the key inequality concerning positive definite elements}
		Let $\mathbb{H}=(A,\delta)$ be a finite quantum hypergroup. Then $v\in A^*$ is positive if and only if $v=\sum_{\alpha\in\Irr(\hat{\mathbb{H}})}\sum_{i,j=1}^{n_\alpha}a^{\alpha}_{ij}v^{\alpha}_{ij}$, where each $[a^{\alpha}_{ij}]_{i,j=1}^{n_\alpha}$ is positive semi-definite and $\{v^{\alpha}=[v^{\alpha}_{ij}]_{i,j=1}^{n_\alpha},\alpha\in\Irr(\hat{\mathbb{H}})\}$ is a complete set of mutually inequivalent, irreducible $^\dagger$-representation of $\hat{\mathbb{H}}$. If this is the case, then we have 
		\begin{equation}\label{key inequality concerning positive definite elements}
		\hat{h}(v)\leq\hat{\epsilon}(v),
		\end{equation}
		where $\hat{h}$ and $\hat{\epsilon}$ are the Haar state and the counit on $\hat{\mathbb{H}}$, respectively.
	\end{cor}
	
	\begin{proof}
		The first part is a consequence of the Lemma \ref{characterization of positive functional}, the discussion at the end of section 1 on finite quantum hypergroups, and the biduality of quantum hypergroups. To show \eqref{key inequality concerning positive definite elements}, assume $\alpha_0\in\Irr(\hat{\mathbb{H}})$ corresponds to the trivial representation, i.e., $v^{\alpha_0}=1$. Then Theorem \ref{Peter-Weyl quantum hyper group} and the positive semi-definiteness of $[a^{\alpha}_{ij}]$ yield:
		\[
		\hat{h}(v)=a^{\alpha_0}\leq\sum_{\alpha\in\Irr(\hat{\mathbb{H}})}\sum_{i=1}^{n_\alpha}a^{\alpha}_{ii}=\hat{\epsilon}(v).
		\]
	\end{proof}
	
	Recall that if $\mathbb{G}$, is finite, $A$ is finite dimensional, so $A=\Pol(\mathbb{G})$ and $\hat{A}=\hat{\mathcal{A}}$. The Fourier transform $\mathcal{F}$ is an isomorphism between Banach *-algebra $(A^*,\Vert\cdot\Vert)$ and finite-dimensional C*-algebra $(\hat{A},\Vert\cdot\Vert)$. 
	
	Let $\phi$ be an idempotent state on finite quantum group $\mathbb{G}=(A,\Delta)$. For any $u\in A^*$ such that $u=u\phi=\phi u$ and $\Vert u-\phi\Vert<1$, define the \textit{logarithm of $u$ with respect to $\phi$} as 
	\[
	\log_\phi(u):=-\sum_{k\geq1}\frac{(\phi-u)^k}{k}.
	\]
	
	Then we have the following properties of logarithm and exponential.
	
	\begin{lem}\label{property of log}
		Suppose that $\mathbb{G}=(A,\Delta)$ is a finite quantum group. Let $\phi$ be an idempotent state on $A$, then for any bounded linear functionals $u,v$ on $A$ such that $u=u\phi=\phi u$ and $v=v\phi=\phi v$, we have
		\begin{enumerate}
			\item $\exp_\phi(\log_\phi(u))=u$, if $\Vert u-\phi\Vert<1$.
			\item $\log_\phi(\exp_\phi(u))=u$, if $\Vert u\Vert<\log 2$.
			\item $\exp_\phi(u+v)=\exp_\phi(u)\exp_\phi(v)$ if $uv=vu$.
			\item $\log_\phi(uv)=\log_\phi(u)+\log_\phi(v)$, if $uv=vu$ and the following holds:
			\[
			\Vert u-\phi\Vert<1,~~\Vert v-\phi\Vert<1,\text{ and }\Vert uv-\phi\Vert<1.
			\]
			\item If moreover, $u$ is a state such that 
			\[
			\Vert u-\phi\Vert<\frac{1}{2} \text{ and } \Vert u^n-\phi\Vert<\frac{1}{2}
			\] for some $n\geq 1$, then 
			\[
			\Vert u^k-\phi\Vert<\frac{1}{2}\text{ for all }1\leq k\leq n.
			\]
			Consequently, in such a case we have 
			\[
			\log_\phi(u^k)=k\log_\phi(u)\text{ for all }1\leq k\leq n.
			\]
		\end{enumerate}
		
	\end{lem}
	
	\begin{proof}
		(1)-(4) are direct and hold on all Banach algebras. To show (5), let $u_0$ be the restriction of $u$ to $A_{\phi}$. Then by Lemma \ref{bi-invariance}, $u_0$ is a state on a finite-dimensional  $C^*$-algebra $A_{\phi}$. Moreover,
		\[
		\Vert u-\phi\Vert
		=\Vert(u_0-\epsilon_0)\mathbb{E}_\phi\Vert
		=\Vert u_0-\epsilon_0\Vert,
		\] 
		where $\epsilon_0$ denotes the restriction of counit $\epsilon$ of $A$ to $A_{\phi}$. Write $A_{\phi}=\oplus_{k=1}^{m}\mathbb{M}_{n_k}(\mathbb{C})$ with matrix units $\{e^k_{ij}:1\leq i,j\leq n_k,1\leq k\leq m\}$. Let $\{\omega^k_{ij}\}$ be its dual basis. By Lemma \ref{characterization of positive functional}, $u_0=\sum_{k=1}^{m}\sum_{i,j=1}^{n_k}b^{(k)}_{ij}\omega^k_{ij}$ with $[b^{(k)}_{ij}]_{i,j=1}^{n_k}\geq0$ and $\Vert u_0\Vert=\sum_{k=1}^{m}\sum_{i=1}^{n_k}b^{(k)}_{ii}=1$. Since $\epsilon_0$ is a character on $A_{\phi}$, there exists $k_0$ such that $n_{k_0}=1$ and $\omega^{k_0}=\epsilon_0$. Thus $u_0-\epsilon_0=(b^{(k_0)}-1)\epsilon+\sum_{k\ne k_0}\sum_{i,j=1}^{n_k}b^{(k)}_{ij}\omega^k_{ij}$ verifies the condition of Corollary \ref{Jordan decomposition} and it follows that
		\[
		\Vert u_0-\epsilon_0\Vert
		=1-b^{(k_0)}+\sum_{k\ne k_0}\sum_{i=1}^{n_k}b^{(k)}_{ii}
		=2-2b^{(k_0)}
		=-2(u_0-\epsilon_0)(e^{k_0}).
		\]
		So for $v_1:=u-\phi$ we have $\Vert v_1\Vert=\Vert w_1\Vert=-2w_1(e^{k_0})$, where $w_1=v_1|_{A_{\phi}}$. Similarly for $v_j:=u^j-\phi$ and $w_j:=v_j|_{A_{\phi}}$ we have
		\begin{equation}\label{computation of norm}
		\Vert v_j\Vert=\Vert w_j\Vert=-2w_j(e^{k_0}),~~j\geq 1.
		\end{equation}
		We show (5) by the induction argument. Clearly it holds for $n=1$. Suppose for now that it holds for $n$. Set $r:=\Vert v_1\Vert$, $s:=\Vert v_n\Vert$ and $t:=\Vert v_{n+1}\Vert$. From $(v_1+\phi)(v_n+\phi)=v_{n+1}+\phi$ it follows that $v_{n+1}=v_1+v_n+v_1 v_n$ and thus by Lemma \ref{bi-invariance} (4) $w_{n+1}=w_1+w_n+w_1 w_n$. Together with \eqref{computation of norm} and Lemma \ref{bi-invariance} (1) one obtains
		\[
		t=r+s-2(w_1 w_n)(e^{k_0})
		\geq r+s-2\Vert w_1 w_n\Vert
		= r+s-2\Vert v_1 v_n\Vert
		\geq r+s-2rs.
		\]
		Then 
		\[
		(1-2r)(1-2s)=4rs-2r-2s+1\geq 1-2t.
		\] 
		By assumption, $1-2r,1-2t>0$, so $1-2s>0$. Hence $u$ and $u^n$ verify the conditions in (4). Using the induction for $n$ we obtain
		\[
		\log_\phi(u^{n+1})=\log_\phi(u)+\log_\phi(u^n)=\log_\phi(u)+n\log_\phi(u)=(n+1)\log_\phi(u).
		\]
		Then the proof for $n+1$ is finished, which shows (5).
	\end{proof}
	
	\begin{rem}
		In fact, to prove (5) we have used the fact that $\Vert u+v\Vert=\Vert u\Vert+\Vert v\Vert$ for all $u,v\in \mathcal{N}_\phi(\mathbb{G})$.
	\end{rem}
	
	\begin{prop}\label{proposition for the first proof}
		Let $\omega$ be an infinitely divisible state on a finite quantum group $\mathbb{G}=(A,\Delta)$. Let $\phi$ be an idempotent state on $A$. Assume that there exists a sequence $\{\omega_{m_j}\}_{j\geq 0}$ of roots of states of $\omega$, with $\omega=\omega_{m_j}^{m_j}$, for all $j$, such that
		\begin{enumerate}
			\item [(1)]$\{m_j\}_{j\geq0}$ is strictly increasing;
			\item [(2)]$\omega_{m_j}=\omega_{m_j}\phi=\phi\omega_{m_j}$ for all $j$;
			\item [(3)]$\omega_{m_j}=\omega_{m_{j+1}}^{n_j}$ for some positive integer $n_j$, $j\geq0$;
			\item [(4)]$\omega_{m_j}\to\phi$, as $j\to\infty$.
		\end{enumerate}
		Then $\omega\in\mathcal{P}_\phi(\mathbb{G})$.
	\end{prop}
	
	\begin{proof}
		Assume that $\{\omega_{m_j}\}$ contains infinitely many different elements, otherwise $\omega=\phi\in\mathcal{P}_\phi(\mathbb{G})$. By $(4)$, we can choose $j_0>0$ such that $\Vert\omega_{m_{j}}-\phi\Vert<1/2$ for all $j\geq j_0$. This inequality, together with $(2)$, allows us to define
		\[
		v_0:=\log_\phi(\omega_{m_{j_0}}), \text{ and }v:=m_{j_0}v_0.
		\]
		Then by the definition of $\omega_{j_0}$ and Lemma \ref{property of log} (1)(2), 
		\[
		\omega=\omega_{m_{j_0}}^{m_{j_0}}
		=\left(\exp_\phi(\log_\phi(\omega_{m_{j_0}}))\right)^{m_{j_0}}
		=\exp_\phi(m_{j_0}v_0)
		=\exp_\phi(v).
		\] 
		To prove $\omega\in\mathcal{P}_\phi(\mathbb{G})$, it suffices to show that $v\in \mathcal{N}_\phi(\mathbb{G})$. For this we check that $v$ verifies Theorem \ref{Theorem 1} (1). Clearly, $v(1)=0$, since $\omega_{j_0}$ is a state. By the definition of logarithm, $v=v\phi=\phi v$. It remains to show that for any $x\in A$ such that $\phi(x^*x)=0$, we have $v_0(x^*x)\geq0$. By $(3)$ we have 
		\[
		\omega_{m_{j_0}}=\omega^{N_j}_j,~~j\geq j_0,
		\]
		where $N_j:=n_{j_0}\cdots n_{j-1}$. Recall that for all $j\geq j_0$, $\Vert\omega_{m_{j}}-\phi\Vert<1/2$. Thus  by Lemma \ref{property of log} (5) we have
		\[
		\log_\phi(\omega_{m_j})=\frac{1}{N_j}\log_\phi(\omega_{m_{j_0}})=\frac{v_0}{N_j},
		\]
		and by Lemma \ref{property of log} (1)
		\[
		\omega_{m_j}=\exp_\phi\left(\frac{v_0}{N_j}\right),~~j\geq j_0.
		\]
		The condition $(1)$ implies that $N_j\to\infty$ as $j\to\infty$. Now for any $x\in A$ such that $\phi(x^*x)=0$, we have  
		\[
		0\leq \omega_{m_j}(x^*x)
		=\phi(x^*x)+\frac{v_0(x^*x)}{N_j}+\sum_{m\geq2}\frac{v_0^m(x^*x)}{N^m_j \cdot m!}
		=\frac{v_0(x^*x)}{N_j}+O(\frac{1}{N^2_j}),
		\]
		for all $j\geq j_0$. Hence
		\[
		v_0(x^*x)+O(\frac{1}{N_j})\geq 0,~~j\geq j_0.
		\]
		Letting $j\to\infty$, we have $v_0(x^*x)\geq0$, which ends the proof.
	\end{proof}
	
	As this proposition suggests, to show that an infinitely divisible state is of Poisson type, it is important to capture the corresponding idempotent state. For this we need two lemmas. The first one is an easy fact in matrix theory.
	
	\begin{lem}\label{lemma in matrix theory}
		Let $P\in\mathbb{M}_n(\mathbb{C})$ be a self-adjoint projection. Suppose $A,B\in\mathbb{M}_n(\mathbb{C})$ such that $A=AP=PA,AB=P$ and $\Vert A\Vert\leq 1$, $\Vert B\Vert\leq 1$. Then $A^*A=AA^*=P$. Consequently, if $u,v$ are states on a finite quantum group $\mathbb{G}$ such that $u=u\phi=\phi u$ and $uv=\phi$, where $\phi$ is an idempotent state on $\mathbb{G}$, then $u^*u=uu^*=\phi$.
	\end{lem}
	
	\begin{proof}
		Since $P$ is a self-adjoint projection, we may assume without loss of generality that 
		\[
		P=\begin{pmatrix*}
		I_r&0\\
		0&0
		\end{pmatrix*},
		\]
		where $I_r$ is the identity in $\mathbb{M}_r(\mathbb{C})$ with $r=\text{rank}(P)$. From $A=AP=PA$ and $AB=P$ it follows 
		\[
		A=\begin{pmatrix*}
		A_r&0\\
		0&0
		\end{pmatrix*},~~
		B=\begin{pmatrix*}
		B_r&*\\
		*&*
		\end{pmatrix*},
		\]
		with $A_r B_r=I_r$. Note that
		\[
		1
		=\Vert I_r\Vert
		=\Vert A_r B_r\Vert
		\leq\Vert A_r\Vert\Vert B_r\Vert
		\leq\Vert A\Vert\Vert B\Vert\leq 1.
		\]
		So $\Vert A_r\Vert=\Vert B_r\Vert=1$. This is to say,
		\[
		\Vert A_r^*A_r\Vert=\Vert B_r B_r^*\Vert=\Vert (A_r^*A_r)^{-1}\Vert=1.
		\]
		Then all the eigenvalues of $A_r^*A_r$ must be $1$ and thus $A_r^*A_r=I_r$. Hence $B_r=A^*_r$  and thus $A^*A=AA^*=P$.
		
		To show the remaining part, note that the Fourier transforms of $u,v$ are in some full matrix algebra $\mathbb{M}_n(\mathbb{C})$. Since $\mathcal{F}$ is a contraction, we have
		\[
		\Vert\hat{u}\Vert\leq\Vert u\Vert=1,~~\Vert\hat{v}\Vert\leq\Vert v\Vert=1.
		\]
		From the fact that the Fourier transform $\mathcal{F}$ is a *-homomorphism it follows
		\[
		\hat{u}=\hat{u}\hat{\phi}=\hat{\phi}\hat{u},~~\hat{u}\hat{v}=\hat{\phi}.
		\]
		Because $\hat{\phi}$ is a self-adjoint  projection, the previous argument yields $\hat{u}^*\hat{u}=\hat{u}\hat{u}^*=\hat{\phi}$. Again, since $\mathcal{F}$ is a *-homomorphism, $\widehat{u^*u}=\widehat{uu^*}=\hat{\phi}$. Thus $u^*u=uu^*=\phi$, by the injectivity of $\mathcal{F}$.
	\end{proof}
	
	\begin{lem}\label{unitary implies root of identity}
		Let $\mathbb{H}=(A,\Delta)$ be a finite quantum hypergroup with the dual $\hat{\mathbb{H}}=(\hat{A},\hat{\Delta})$, which is also a finite quantum hypergroup. Suppose that $u$ is a state on $A$ such that $uu^*=u^*u=\epsilon$, where $\epsilon$ is the counit of $A$. Then $u$ is an $n$-th root of $\epsilon$ for some $n\leq\dim\hat{A}$. If, moreover, $\mathbb{H}$ is a finite quantum group, then $u$ is also a character.
	\end{lem}

	\begin{proof}
		Let $v^\alpha=[v^\alpha_{ij}]_{i,j=1}^{n_\alpha}$ be the representation corresponding to $\alpha\in\Irr(\hat{\mathbb{H}})$. By Corollary \ref{Corollay about the key inequality concerning positive definite elements} we can write $u$ as 
		\[
		u=\sum_{\alpha\in\Irr(\hat{\mathbb{H}})}p_\alpha\sum_{i,j=1}^{n_\alpha}a^\alpha_{ij}v^\alpha_{ij},
		\]
		where $p_\alpha\geq0$, $\sum_{\alpha\in\Irr(\hat{\mathbb{H}})}p_\alpha=1$ and $[a^\alpha_{ij}]_{i,j=1}^{n_\alpha}$ positive semi-definite with trace $1$. 
		Let $\alpha_0$ be the trivial representation, i.e., $v_{\alpha_0}=1$. Then we can write $u=p_{\alpha_0}1+v$ with $v=\sum_{\alpha\ne\alpha_0}p_\alpha\sum_{i,j=1}^{n_\alpha}a^\alpha_{ij}v^\alpha_{ij}$ such that $\Vert v\Vert=\sum_{\alpha\ne\alpha_0}p_\alpha=1-p_{\alpha_0}$, where $\Vert\cdot\Vert$ denotes the norm of $v$ as a functional. We have $uu^*=p^2_{\alpha_0}1+p_{\alpha_0}(v+v^*)+vv^*$.
		Then by Theorem \ref{Peter-Weyl quantum hyper group}, $\hat{h}(v)=\hat{h}(v^*)=0$. Since $vv^*$ is also a positive functional, from Corollary \ref{Corollay about the key inequality concerning positive definite elements} it follows
		\[
		\hat{h}(vv^*)\leq \hat{\epsilon}(vv^*)=|\hat{\epsilon}(v)|^2=(1-p_{\alpha_0})^2,
		\]
		where $\hat{h}$ and $\hat{\epsilon}$ are the Haar state and the counit on $\hat{\mathbb{H}}$, respectively.
		So we have 
		\[
		1=\hat{h}(uu^*)=p_{\alpha_0}^2+\hat{h}(vv^*)\leq p_{\alpha_0}^2+(1-p_{\alpha_0})^2=1-2p_{\alpha_0}+2p_{\alpha_0}^2,
		\]
		which yields $p_{\alpha_0}(p_{\alpha_0}-1)\geq 0$. Recall that $0\leq p_{\alpha_0}\leq 1$, hence either $p_{\alpha_0}=0$ or $p_{\alpha_0}=1$. That is to say, either $\hat{h}(u)=0$ or $u=\epsilon$. Since for any $n\geq1$, $u^n$ is again a state such that $u^n u^{*n}=\epsilon$, we obtain, by a similar argument, that either $\hat{h}(u^n)=0$ or $u^n=\epsilon$.
		
		If $u$ is not a $n$-th root of $\epsilon$ for all $1\leq n\leq\dim\hat{A}$, then we have 
		\[
		\hat{h}(u^n)=0,~~n=1,2,\dots,\dim \hat{A}.
		\]
		Set $P(\lambda):=\det(\lambda I_m-u)=\sum_{i=0}^{m}a_i \lambda^i$, then Cayley-Hamilton Theorem implies that $P(u)=0$, where $m\leq \dim \hat{A}$ is a positive integer and $I_m$ denotes the identity matrix in $\mathbb{M}_{m}(\mathbb{C})$. Since $u$ is unitary in $\mathbb{M}_{m}(\mathbb{C})$, we have $a_0=(-1)^m\det(u)\neq0$. But 
		\[
		a_0=a_0\hat{h}(I_m)+\sum_{i=1}^{m}a_i\hat{h}( u^i)=\hat{h}(P(u))=0,
		\]
		which leads to a contradiction. So we must have $u^m=\epsilon$ for some $1\leq m\leq \dim \hat{A}$.
		
		If moreover, $\mathbb{H}$ is a finite quantum group, then we can obtain a slightly stronger estimate. Indeed, by choosing $\{v^\alpha_{ij}\}$ to be unitary irreducible representations, we have from the orthogonal relation \eqref{Orthogonal relation for finite quantum groups} that
		\[
		1=\hat{h}(uu^*)
		=\sum_{\alpha\in\Irr(\hat{\mathbb{H}})}\frac{p^2_\alpha}{n_\alpha}\sum_{i,j=1}^{n_\alpha}|a^\alpha_{ij}|^2
		\leq\sum_{\alpha\in\Irr(\hat{\mathbb{H}})}\frac{p^2_\alpha}{n_\alpha}
		\leq\sum_{n_\alpha=1}p^2_\alpha+\sum_{n_\alpha\geq 2}\frac{p^2_\alpha}{2},
		\]
		where the first inequality holds because $[a^\alpha_{ij}]_{i,j=1}^{n_\alpha}$ is positive semi-definite with trace $1$.
		Recall that $p_\alpha\geq0$ and $\sum_{\alpha\in\Irr(\hat{\mathbb{H}})}p_\alpha=1$, thus
		\[
		1\leq\sum_{n_\alpha=1}p^2_\alpha+\sum_{n_\alpha\geq 2}\frac{p^2_\alpha}{2}
		\leq\sum_{n_\alpha=1}p^2_\alpha+\sum_{n_\alpha\geq 2}p^2_\alpha
		\leq\sum_{\alpha\in\Irr(\hat{\mathbb{H}})}p_\alpha
		=1.
		\]
		This happens only if $p_{\alpha'}=1$ and $n_{\alpha'}=1$ for some $\alpha'\in\Irr(\hat{\mathbb{H}})$. That is to say, $u=v^{\alpha'}$ is a one dimensional unitary representation of $\hat{\mathbb{H}}$, thus a character.
	\end{proof}
	
	The following proposition, gathering the main ingredients of preceding lemmas, will be used to prove Theorem \ref{main thm}.
	
	\begin{prop}\label{prop before the final thm}
		Let $\mathbb{G}=(A,\Delta)$ be a finite quantum group with the counit $\epsilon$. Suppose that $u,v\in\mathcal{S}(A)$ and $\phi\in\text{Idem}(\mathbb{G})$ are such that  $u=u\phi=\phi u$ and $uv=\phi$. Then there exists a positive integer $m\leq \dim \hat{A}$ such that $u^m=\phi$.
	\end{prop}
	
	\begin{proof}
		From Lemma \ref{lemma in matrix theory} it follows $u^*u=uu^*=\phi$. Let $u_0$ and $\epsilon_0$ be the restrictions of $u$ and $\epsilon$ to $A_{\phi}$, respectively. Then $u_0$ is a state on finite quantum hypergroup $\mathbb{H}:=(A_{\phi},\Delta_{\phi})$, and $\epsilon_0$ is the counit on $\mathbb{H}$ such that $u_0 u^*_0=\epsilon_0$. Note that by Theorem \ref{Duality induced by an idempotent state}, the dual of $\mathbb{H}$ is $(\hat{A}_p,\hat{\Delta}_p)$, which is again a finite quantum hypergroup, where $p=\phi$ is considered as a group-like projection in $\hat{A}$.  So Lemma \ref{unitary implies root of identity} implies $u^m_0=\epsilon_0$ for some $m\leq\dim\widehat{A_\phi}\leq\dim\hat{A}$. Hence $u^m=u^m_0\mathbb{E}_\phi=\phi$. 
	\end{proof}
	
	Now we are ready to prove the main result of this paper.
	
	\begin{thm}\label{main thm}
		Let $\mathbb{G}=(A,\Delta)$ be a finite quantum group. Then $\mathcal{P}(\mathbb{G})=\mathcal{I}(\mathbb{G})$. 
	\end{thm}
	
	\begin{proof}[The first proof]
		$\mathcal{P}(\mathbb{G})\subset\mathcal{I}(\mathbb{G})$ is clear. Let $\omega\in\mathcal{I}(\mathbb{G})$. We claim that for any positive integer $N\geq 2$, there exists a sequence $\{b_n\}_{n\geq 0}$ of roots of $\omega$ such that $b_0=\omega,b_{n-1}=b^N_n,n\geq1$. Indeed, since $A$ is finite dimensional, the set of states $Z=\mathcal{S}(A)$ is compact with respect to the norm topology. Thus $\prod_{j\geq 0}Z_j$, where $Z_j=Z$ for all $j$, is compact with respect to the product topology. Let $a_n\in Z$ be any $n$-th root of $\omega$ for all $n\geq 0$. Then the sequence of non-empty closed sets
		\[
		W_k:=\overline{\bigcup}_{j\geq k}\{a^{N^j}_{N^j}\}\times\{a^{N^{j-1}}_{N^j}\}\times\cdots\times\left\{a_{N^j}\right\}\times\prod_{i\geq j}Z_i,~~k\geq 1,
		\]
		is decreasing: $W_1\supset W_2\supset\cdots$, and thus any finite intersection of $\{W_k\}_{k\geq1}$ is non-empty. By compactness of $\prod_{j\geq 0}Z_j$, $\bigcap_{k\geq 1}W_k\neq\emptyset$. Hence one can choose $(b_0,b_1,\dots)\in\bigcap_{k\geq 1}W_k$, which verifies
		\[
		b_0=\omega,~~b_{n-1}=b^{N}_n,~~n\geq 1.
		\]
		This proves the claim.
		
		Choose $N=(\dim \hat{A})!\geq 2$ and let $\{b_n\}_{n\geq 0}$ be as above. Since $Z$ is compact, there exists a subsequence $\{c_j\}_{j\geq0}$ of $\{b_i\}_{i\geq0}$ such that $c_j$ converges to some $c\in Z$. If we fix a non-negative integer $i$, we have $b_i=c^{r_j}_j$ for sufficient large $j$ and some integer $r_j\geq N\geq 2$. That is,
		\begin{equation}\label{relation 1 of roots}
		b_i=c_j c^{r_j-1}_j=c^{r_j-1}_j c_j.
		\end{equation}
		We can assume that $c^{r_j-1}_j$ converges to some $d_i\in Z$, otherwise consider some subsequence, since  $\{c^{r_j-1}_j\}_{j\geq0}\subset Z$. Thus letting $j\to\infty$ in \eqref{relation 1 of roots}, we have
		\begin{equation}\label{relation 2 of roots}
		b_i=c d_i=d_i c,~~i\geq 0.
		\end{equation}
		This implies $b_i\in cZ\cap Zc$ for all $i\geq0$. From the choice of $c_j$ we have $c_j\in cZ\cap Zc$ for all $j\geq0$. Then for any $i$ the corresponding $c^{r_j-1}_j\in cZ\cap Zc$ for all $j$, which implies that $d_i\in cZ\cap Zc$ by the compactness of $cZ\cap Zc$. Now considering \eqref{relation 1 of roots} for $\{c_j\}_{j\geq0}$, instead of $\{b_i\}_{i\geq0}$, we obtain an updated version of \eqref{relation 2 of roots}:
		\begin{equation}\label{relation for the idempotent state}
		c_j=c d'_j=d'_j c,~~j\geq 0,
		\end{equation}
		where $d'_j\in cZ\cap Zc$. Letting $j\to\infty$, consider the subsequence of $\{d'_j\}_{j\geq0}$ if necessary, one obtains
		\begin{equation}\label{relation 3 of roots}
		c=cd=dc,
		\end{equation}
		where $d\in cZ\cap Zc$ by the compactness of $cZ\cap Zc$. Suppose $d=ce$ for some $e\in Z$, then $d^2=dce=ce=d$, i.e., $d$ is an idempotent state. By Proposition \ref{prop before the final thm}, we obtain $c^m=d$ for some $m\leq\dim \hat{A}$. 
		Then by choosing $N$ to be $(\dim \hat{A})!$, we have 
		\[
		c^N_j\to c^N=(c^m)^{\frac{N}{m}}=d, \text{ as }j\to\infty.
		\]
		Denote by $\phi$ the idempotent state $d$. Set $\omega_0:=\omega$ and $\omega_n:=c_n^N$ for all $n\geq1$. Then $\omega_n\to\phi$ as $n$ tends to $\infty$.
		By definition, $\{\omega_n\}_{n\geq0}$ is a subsequence of $\{b_j\}_{j\geq0}$, thus $\omega_{n-1}=\omega_n^{s_n}$ with $s_n=N r_n$ for all $n\geq 1$. Moreover, from \eqref{relation for the idempotent state} we have 
		\[
		\omega_n=c^N_n=(c d'_n)^N=c^N d_n^{'N}=\phi d_n^{'N}=\phi(\phi d_n^{'N})=\phi\omega_n,~~n\geq0.
		\]
		Similarly, $\omega_n=\omega_n\phi,n\geq0.$ Hence $\{\omega_n\}_{n\geq0}$ verifies the conditions of Proposition \ref{proposition for the first proof}, and consequently $\omega\in\mathcal{P}_\phi(\mathbb{G})$.
	\end{proof}
	
	Before giving the second proof, we introduce the following proposition, which could be formulated and proved  for a general Banach algebra.
	
	\begin{prop}\label{prop for second proof}
		Let $\mathbb{G}=(A,\Delta)$ be a compact quantum group, with $A$ separable. Let $\omega$ be a infinitely divisible state on $\mathbb{G}$. Suppose that there exist an idempotent state $\phi$ and a sequence of $\phi$-bi-invariant roots $\{\omega_{n_k}\}_{k\geq1}$ of $\omega$, where $\{n_k\}_{k\geq 0}$ is an increasing sequence of positive integers, such that $\omega_{n_k}^{n_k}=\omega$ for all $k\geq 1$, and
		\begin{equation}\label{decay condition}
		\sup_{k\geq1}n_k\Vert \omega_{n_k}-\phi\Vert=M<\infty,
		\end{equation}
		then $\omega\in\mathcal{P}_\phi(\mathbb{G})$.
	\end{prop}
	
	\begin{proof}
		Recall that $\{\varphi\in A^*:\Vert\varphi\Vert\leq M\}$ is compact with respect to weak* topology for each $M>0$. Then from \eqref{decay condition} we have for some subsequence of $\{n_k\}_{k\geq1}$, still denoted by $\{n_k\}_{k\geq1}$, that $n_k(\omega_{n_k}-\phi)$ converges to an element $u\in A^*$ with respect to the weak$^*$ topology. Then $u=\lim\limits_{k\to\infty}n_k(\omega_{n_k}-\phi)\in\mathcal{N}_\phi(\mathbb{G})$ and $\exp_\phi(u)$ is a Poisson state. It suffices to show that $\omega=\exp_\phi(u)$. Set 
		\[
		u_{n_k}:=\sum_{m\geq2}\frac{(\omega_{n_k}-\phi)^m}{m!},~~k\geq 1.
		\]
		It is well-defined, since $\sup_{k\geq 1}\Vert \omega_{n_k}-\phi\Vert<\infty$. Moreover, from \eqref{decay condition} it follows
		\[
		\Vert u_{n_k}\Vert\leq\sum_{m\geq2}\frac{1}{m!}\left(\frac{M}{n_k}\right)^m\leq\frac{1}{n_k^2}\sum_{m\geq2}\frac{M^m}{m!},
		\]
		whence
		\[
		\lim\limits_{k\to\infty}n_k\Vert u_{n_k}\Vert=0,
		\]
		and 
		\[
		\lim\limits_{k\to\infty}\left(1+\Vert u_{n_k}\Vert\right)^{n_k}=\lim\limits_{k\to\infty}\left(1+\Vert u_{n_k}\Vert\right)^{\frac{1}{\Vert u_{n_k}\Vert}\cdot n_k\Vert u_{n_k}\Vert}=1.
		\]
		Hence
		\begin{align*}
		\Vert \omega-\exp_\phi(n_k(\omega_{n_k}-\phi))\Vert
		&=\Vert \omega_{n_k}^{n_k}-(\omega_{n_k}+u_{n_k})^{n_k}\Vert\\
		&\leq\sum_{m=1}^{n_k}\binom{n_k}{m}\Vert \omega_{n_k}^{n_k-m}u_{n_k}^m\Vert\\
		&\leq\sum_{m=1}^{n_k}\binom{n_k}{m}\Vert u_{n_k}\Vert^m
		=\left(1+\Vert u_{n_k}\Vert\right)^{n_k}-1,
		\end{align*}
		which tends to $0$ as $k\to\infty$. This shows $\omega=\exp_\phi(u)$ and finishes the proof.
	\end{proof}	
	
	\begin{proof}[The second proof of Theorem \ref{main thm}]
		Again, $\mathcal{P}(\mathbb{G})\subset\mathcal{I}(\mathbb{G})$ is clear. Let $\omega\in\mathcal{I}(\mathbb{G})$. From the first proof we know that there exist an idempotent state $\phi\in\text{Idem}(\mathbb{G})$ and a sequence of roots $\{\omega_{n_k}\}_{k\geq0}\subset\mathcal{S}(A)$ with $\{n_k\}_{k\geq 0}$ an increasing sequence of positive integers such that 
		\[
		\omega_{n_k}^{n_k}=\omega,~~ \omega_{n_k}=\omega_{n_k}\phi=\phi\omega_{n_k},~~k\geq 0,
		\]
		and $\omega_{n_k}\to\phi$ as $k\to\infty$. Let $u$ and $u_{n_k}$ be the restrictions of $\omega$ and $\omega_{n_k}$ to $A_{\phi}$ for all $k\geq0$, respectively. Then from Lemma \ref{bi-invariance} $u$ is a state on finite quantum hypergroup $\mathbb{H}=(A_{\phi},\Delta_{\phi})$ such that $\{u_{n_k}\}_{k\geq0}$ is a sequence of roots of $u$ in $\mathcal{S}(A_{\phi})$ verifying
		\[
		u_{n_k}^{n_k}=u \text{ and }u_{n_k}\to\epsilon_0,~~ k\to\infty,
		\]
		where $\epsilon_0$ is the counit of $\mathbb{H}$. Now we repeat a calculation in Lemma \ref{unitary implies root of identity}. Let $\Irr(\hat{\mathbb{H}})$ be the set of unitary equivalent classes of irreducible unitary representations of $\hat{\mathbb{H}}$. Let $v^\alpha_{ij}$ be the matrix elements corresponding to the representation. Then we can write $u_{n_k}$ as 
		\[
		u_{n_k}=\sum_{\alpha\in\Irr(\hat{\mathbb{H}})}p_{\alpha,k}\sum_{i,j=1}^{n_\alpha}a^{\alpha,k}_{ij}v^\alpha_{ij}
		\] 
		with $p_{\alpha,k}\geq0$, $\sum_{\alpha\in\Irr(\hat{\mathbb{H}})}p_{\alpha,k}=1$ and $[a^\alpha_{ij}]_{i,j=1}^{n_\alpha}$ and $[a^{\alpha,k}_{ij}]_{i,j=1}^{n_\alpha}$ positive semi-definite with trace $1$ for each $k$.
		
		Denote by $\alpha_0$ the trivial representation in $\Irr(\hat{\mathbb{H}})$, so that $\epsilon_0=v^{\alpha_0}$. By Corollary \ref{Jordan decomposition} and the assumption,
		\begin{equation}\label{distance to counit}
		\Vert u_{n_k}-\epsilon_0\Vert = 2(1-p_{\alpha_0,k})\to 0,~~k\to\infty.
		\end{equation}
		So $p_{\alpha_0,k}\to 1$ as $k\to\infty$. 
		
		Let $m:=\text{rank}(\hat{\epsilon_0})$. Then $\widehat{A_{\phi}}$ can be viewed as a subalgebra of $\mathbb{M}_m(\mathbb{C})$. We use $\Vert\cdot\Vert_p$ to denote the Schatten $p$-norm of matrices. By H\"older's inequality,
		\begin{equation}\label{Holder}
		\Vert \hat{u}\Vert_{2/n_k}=\Vert \hat{u_{n_k}}^{n_k}\Vert_{2/n_k}\leq\Vert\hat{u_{n_k}}\Vert_{2}^{n_k},~~k\geq 0.
		\end{equation}
		
		Let $\lambda_1,\dots,\lambda_m$ be all singular values of $\hat{u}$. Then all $\lambda_i$ are non-zero. To see this, it suffices to show that $\hat{u}$ is invertible. Note that $\hat{\epsilon_0}$ is the identity matrix in $\mathbb{M}_m(\mathbb{C})$. Since for large $k$ there holds
		$$\Vert \hat{u_{n_k}}-\hat{\epsilon_0}\Vert\leq \Vert u_{n_k}-\epsilon_0\Vert<1,$$ 
		we have that $\hat{u_{n_k}}$ is invertible for large $k$, and so is $\hat{u}$. 
		
		Following a similar calculation to that in  Lemma \ref{unitary implies root of identity}, Theorem \ref{Peter-Weyl quantum hyper group} and Corollary \eqref{key inequality concerning positive definite elements} imply
		\[
		\hat{h}(u_{n_k}u_{n_k}^*)
		=p_{\alpha_0,k}^2+\hat{h}(v_{n_k}v_{n_k}^*)
		\leq p_{\alpha_0,k}^2+\hat{\epsilon}(v_{n_k}v_{n_k}^*)
		=p_{\alpha_0,k}^2+(1-p_{\alpha_0,k})^2,
		\]
		where $v_{n_k}=u_{n_k}-p_{\alpha_0,k}\epsilon_0=\sum_{\alpha\ne\alpha_0}p_{\alpha,k}\sum_{i,j=1}^{n_\alpha}a^{\alpha,k}_{ij}v^\alpha_{ij}$.
		This, together with \eqref{distance to counit} and \eqref{Holder}, yields
		\[
		\frac{1}{m}\sum_{i=1}^{m}\lambda_i^{2/n_k}
		=\Vert \hat{u}\Vert_{2/n_k}^{2/n_k}
		\leq\Vert u_{n_k}\Vert_{2}^2
		\leq p_{\alpha_0,k}^2+(1-p_{\alpha_0,k})^2,
		\]
		for all $k$. Since $p_{\alpha_0,k}\to 1$ as $k\to\infty$,  there exists $K>0$ such that for all $k\geq K$, $\frac{1}{2}\leq p_{\alpha_0,k}\leq 1$. Thus for all $k\geq K$, $1-p_{\alpha_0,k}\leq p_{\alpha_0,k}$ and then
		\begin{equation}\label{key inequality in second proof}
		\frac{1}{m}\sum_{i=1}^{m}\lambda_i^{2/n_k}
		\leq p_{\alpha_0,k}^2+(1-p_{\alpha_0,k})^2
		\leq p_{\alpha_0,k}^2+p_{\alpha_0,k}(1-p_{\alpha_0,k})
		=p_{\alpha_0,k}.
		\end{equation}
		Combining this with \eqref{distance to counit}, we have
		\[
		n_k\Vert u_{n_k}-\epsilon_0\Vert
		=2n_k(1-p_{\alpha_0,k})
		\leq 2n_k(1-\frac{1}{m}\sum_{i=1}^{m}\lambda_i^{2/n_k})\leq M<\infty,
		\]
		for all $k\geq K$, where $M$ is a constant independent of $k$. Here we have used the fact that $\lambda_i>0$ for all $i$. From Lemma \ref{bi-invariance} it follows that
		\[
		\sup_{k\geq 1}n_k\Vert \omega_{n_k}-\phi\Vert=\sup_{k\geq 1}n_k\Vert u_{n_k}-\epsilon_0\Vert<\infty.
		\]
		Then $\omega\in\mathcal{P}_\phi(\mathbb{G})$ by Proposition \ref{prop before the final thm}.
	\end{proof}
	
	\begin{rem}
		Both proofs rely on the capture of idempotent state where the infinitely divisible state is ``supported on" and the sequence of roots converging to this idempotent state. After this the first proof aims to show that this sequence of roots can chosen to form a \emph{submonogeneous convolution semigroup} (Proposition \ref{proposition for the first proof} (3)), while the idea of the second proof is derived from a general result Proposition \ref{prop for second proof}, concerning the decay property \eqref{key inequality in second proof} of this sequence of roots. The inequality \eqref{key inequality in second proof} also allows us to simplify the proof of the main theorem in \cite{KRP} for the finite case.
	\end{rem}
	
	\subsection*{Acknowledgment}
	The author would like to thank his advisor Adam Skalski for bringing him the topic, helpful discussions and constant encouragement. A. Skalski also checked an earlier version of this paper carefully and gave many useful comments. This paper could not be finished without his help. The research was partially supported by the NCN (National Centre of Science) grant 2014/14/E/ST1/00525 and the French ``Investissements d'Avenir" program, project ISITE-BFC (contract ANR-15-IDEX-03).


\begin{thebibliography}{DVD11b}
		
		\bibitem[B\"og59]{Boge}
		W. B\"oge.
		\newblock \"Uber die {C}harakterisierung unendlich teilbarer
		{W}ahrscheinlichkeitsverteilungen.
		\newblock {\em J. Reine Angew. Math.}, 201:150--156, 1959.
		
		\bibitem[CV99]{Compactquantumhypergroup}
		Y.~A. Chapovsky and L.~I. Vainerman.
		\newblock Compact quantum hypergroups.
		\newblock {\em J. Operator Theory}, 41(2):261--289, 1999.
		
		\bibitem[DFW17]{preprintofXumin}
		B.~Das, U.~Franz, and X.~Wang.
		\newblock Invariant Markov semigroups on quantum homogeneous spaces.
		\newblock {\em Preprint}, 2017.
		
		\bibitem[DVD11a]{Algebraicquantumhypergroups}
		L.~Delvaux and A.~Van~Daele.
		\newblock Algebraic quantum hypergroups.
		\newblock {\em Adv. Math.}, 226(2):1134--1167, 2011.
		
		\bibitem[DVD11b]{constructionsandexamplesofquantumhypergroups}
		L.~Delvaux and A.~Van~Daele.
		\newblock Algebraic quantum hypergroups {II}. {C}onstructions and examples.
		\newblock {\em Internat. J. Math.}, 22(3):407--434, 2011.
		
		\bibitem[FK17]{Group-likeprojectionsforLCQG}
		R.~{Faal} and P.~{Kasprzak}.
		\newblock {Group-like projections for locally compact quantum groups}.
		\newblock {arXiv:1706.10138}, June 2017.
		
		\bibitem[FS09a]{Newcharacterisationofidempotentstates}
		U. Franz and A. Skalski.
		\newblock A new characterisation of idempotent states on finite and compact
		quantum groups.
		\newblock {\em C. R. Math. Acad. Sci. Paris}, 347(17-18):991--996, 2009.
		
		\bibitem[FS09b]{Idempotentstatesoncqg}
		U. Franz and A. Skalski.
		\newblock On idempotent states on quantum groups.
		\newblock {\em J. Algebra}, 322(5):1774--1802, 2009.
		
		\bibitem[Kal01]{Newexamples}
		A.~A. Kalyuzhnyi.
		\newblock Conditional expectations on compact quantum groups and new examples
		of quantum hypergroups.
		\newblock {\em Methods Funct. Anal. Topology}, 7(4):49--68, 2001.
		
		\bibitem[KI40]{KawadaIto}
		Y. Kawada and K. It\^o.
		\newblock On the probability distribution on a compact group. {I}.
		\newblock {\em Proc. Phys.-Math. Soc. Japan (3)}, 22:977--998, 1940.
		
		\bibitem[KS18]{LatticeofidempotentstatesonLCQG}
		P.~{Kasprzak} and P.~M. {So{\l}tan}.
		\newblock {The Lattice of Idempotent States on a Locally Compact Quantum
			Group}.
		\newblock {arXiv:1802.03953}, February 2018.
		
		\bibitem[LVD08]{grouplikeprojections}
		M.~B. Landstad and A.~Van~Daele.
		\newblock Groups with compact open subgroups and multiplier {H}opf
		{$*$}-algebras.
		\newblock {\em Expo. Math.}, 26(3):197--217, 2008.
		
		\bibitem[MVD98]{CompactquantumgroupsVanDaele}
		A. Maes and A. Van~Daele.
		\newblock Notes on compact quantum groups.
		\newblock {\em Nieuw Arch. Wisk. (4)}, 16(1-2):73--112, 1998.
		
		\bibitem[Pal96]{Pal}
		A. Pal.
		\newblock A counterexample on idempotent states on a compact quantum group.
		\newblock {\em Lett. Math. Phys.}, 37(1):75--77, 1996.
		
		\bibitem[Par70]{KRP}
		K.~R. Parthasarathy.
		\newblock Infinitely divisible representations and positive definite functions
		on a compact group.
		\newblock {\em Comm. Math. Phys.}, 16:148--156, 1970.
		
		\bibitem[Sch72]{poissonlaws}
		L.~Schmetterer.
		\newblock On {P}oisson laws and related questions.
		\newblock pages 169--185, 1972.
		
		\bibitem[VD96]{Discretequantumgroups}
		A.~Van~Daele.
		\newblock Discrete quantum groups.
		\newblock {\em J. Algebra}, 180(2):431--444, 1996.
		
		\bibitem[Wor98]{CompactquantumgroupsWoronowicz}
		S.~L. Woronowicz.
		\newblock Compact quantum groups.
		\newblock In {\em Sym\'etries quantiques ({L}es {H}ouches, 1995)}, pages
		845--884. North-Holland, Amsterdam, 1998.
		
	\end{thebibliography}
\end{document}